\def\PaperDate{July 19, 2023}
\documentclass[12pt,a4paper]{article}
  \pdfoutput=1
  \newcommand{\email}[1]{#1}
  \usepackage{amsmath}
  \usepackage{amsfonts}
  \usepackage{amsthm}
  \usepackage{ifthen}
  \expandafter\let\expandafter\IfThenElse\csname ifthenelse\endcsname
  \usepackage{tikz}
  \usepackage{tikz-cd}
  \usetikzlibrary{calc}
  \usepackage{MnSymbol}
  \usepackage[backend=bibtex,style=alphabetic, sorting=nyt,  giveninits=false, eprint=false, url=false, doi=false, isbn=false, related=false, arxiv=pdf]{biblatex}
  \renewbibmacro{in:}{}
  \bibliography{Trees-bibliography}

  \newtheorem{globalthm}{NotToBeUsedDirectly}[section]
  \newtheorem{proposition}[globalthm]{Proposition}
  \newtheorem{lemma}[globalthm]{Lemma}
  \newtheorem{observation}[globalthm]{Observation}
  
  \newtheorem{theorem}[globalthm]{Theorem}
  \newtheorem{corollary}[globalthm]{Corollary}
  \theoremstyle{definition}

  \newtheorem{remark}[globalthm]{Remark}
  \usepackage{xcolor}
  \definecolor{darkgreen}{rgb}{0.0, 0.5, 0.0}

  \newcommand{\pref}[1]{(\ref{#1})}
  \newcommand{\notion}[1]{\emph{#1}}
  
  \newlength{\vshrinkamount}
  \newcommand{\vshrink}[2][0pt]{%
    \begingroup
    \setlength\fboxrule{0pt}%
    \setlength\fboxsep{-\vshrinkamount}%
    \framebox[\width]{$#2$}%
    \endgroup
  }

  \newcommand{\IfEmptyThenElse}[1]{%
    \begingroup
    \def\dummy{#1}%
    \def\empty{}%
    \ifx\dummy\empty
      \def\next##1##2{##1}%
    \else
      \def\next##1##2{##2}%
    \fi
    \expandafter\endgroup\next
  }
  \newcommand{\OptSubscript}[1][]{\IfEmptyThenElse{#1}{}{_{#1}}}
  \newcommand{\OptExponent}[1][]{\IfEmptyThenElse{#1}{}{^{#1}}}
  \newcommand{\AddParentheses}[1]{(#1)}
  
  \newcommand{\ExpPar}[1][]{\OptExponent[#1]\AddParentheses}
  
  \newcommand{\SubExp}[1][]{\OptSubscript[#1]\OptExponent}
  \newcommand{\SubExpPar}[1][]{\OptSubscript[#1]\ExpPar}
  
  \newcommand{\with}{\mid}
  \makeatletter
  \newcommand{\newvariable}[2]{%
    \newcommand{#1}{#2\SubExp}
    \expandafter\newcommand\csname\expandafter\@gobble\string#1Of\endcsname{#2\SubExpPar}
  }
  \makeatother
  \newcommand{\SetOf}[2][]{\IfEmptyThenElse{#1}{
      \left\{\vshrink{#2}\right\}
    }{
      \left\{\,\vshrink{#1\with #2\,}\right\}
    }
  }
  \newcommand{\BlockOf}[1]{\left[\,\vshrink{#1}\,\right]}

  \newcommand{\CopyArg}[1]{#1}
  \newcommand{\ExpCopy}[1][]{^{#1}\CopyArg}
  \newcommand{\SubExpCopy}[1][]{_{#1}\ExpCopy}
  \newcommand{\IndexOp}[1]{#1\SubExpCopy}
  \newcommand{\Sum}{\IndexOp{\sum}}
  \newcommand{\Int}{\IndexOp{\int}}
  \newcommand{\Max}{\IndexOp{\max}}

  \newcommand{\Lim}{\IndexOp{\lim}}
  
  \newvariable{\RealNumbers}{\mathbb{R}}
  \newvariable{\ComplexNumbers}{\mathbb{C}}
  \newvariable{\NatNumbers}{\mathbb{N}}

  \newcommand{\crossprod}{\times}
  \newcommand{\mapcolon}{:\,}
  \newvariable{\TheDegree}{\operatorname{deg}}
  \newvariable{\ModDegree}{q}
  \newvariable{\TheGraph}{\mathcal{G}}
  \newvariable{\VertexSet}{\mathcal{X}}
  \newvariable{\TheVertex}{x}
  \newvariable{\AltVertex}{y}
  \newvariable{\EdgeSet}{\mathcal{E}}
  \newvariable{\Terminal}{\tau}
  \newvariable{\Initial}{\iota}
  \newcommand{\Opposite}{\operatorname{op}}
  \newcommand{\OppositeOf}[1]{#1^{\Opposite}}
  \newvariable{\TheEdge}{e}
  \newvariable{\AltEdge}{f}
  \newvariable{\PathSet}{\mathcal{P}}
  \newvariable{\ThePath}{p}
  \newvariable{\AltPath}{q}
  \newvariable{\TheIndex}{i}
  \newvariable{\AltIndex}{j}
  \newvariable{\ThrIndex}{k}
  \newvariable{\Projection}{\pi}

  \newcommand{\adj}{-}
  \newcommand{\turn}{\sqfrown}

  \newcommand{\MapsFromTo}[2]{#2^{#1}}
  \newcommand{\ExplicitMapsFromTo}[3][]{\operatorname{Maps}_{#1}(#2;#3)}
  \newcommand{\MapsOn}[1]{\MapsFromTo{#1}{\ComplexNumbers}}
  \newcommand{\compose}{\circ}
  \newcommand{\Equalizer}[2]{\SetOf{#1=#2}}
  \newcommand{\op}{\mathrm{op}}
  \newcommand{\subspace}{\subseteq}

  \newvariable{\TheTwist}{z}
  \newvariable{\NeighborSum}{\Sigma}
  \newvariable{\Rescale}{\sigma}
  \newvariable{\NeighborAvg}{\Delta}
  \newvariable{\LocalTweak}{\delta}
  \newvariable{\TurnSum}{S}
  \newvariable{\Identity}{\operatorname{Id}}
  \newvariable{\Inclusion}{I}
  \newvariable{\Tf}{\mathcal{L}}
  \newvariable{\DualTf}{\Tf'}
  \newvariable{\Gradient}{\nabla}
  \newvariable{\InSum}{O}

  \newvariable{\FnVert}{v}
  \newvariable{\FnEdge}{X}
  \newvariable{\TheFnEdge}{X}
  \newvariable{\AltFnEdge}{Y}
  \newvariable{\FnPath}{u}

  \newcommand{\BoundedMaps}[2][]{\mathcal{B}^{#1}(#2;\ComplexNumbers)}
  \newcommand{\AbsValueOf}[1]{|#1|}
  \newcommand{\NormOf}[2][]{\|#2\|_{#1}}
  \newvariable{\Depends}{F}
  \newvariable{\TheLength}{n}
  \newvariable{\TheDegBound}{D}

  \newcommand{\LengthOf}[1]{\ell(#1)}
  \newcommand{\DummyArg}{-}
  \newcommand{\Distance}[1][]{d_{#1}}
  \newcommand{\DistanceOf}[3][]{\Distance[#1](#2,#3)}
  \newcommand{\InfLipConstOf}[2][]{|#2|_{#1}}
  \newcommand{\LipNormOf}[2][]{\|#2\|_{#1}}
  \newcommand{\InnerConstOf}[2][]{|#2|^{1}_{#1}}
  \newcommand{\InnerNormOf}[2][]{\|#2\|^{1}_{#1}}
  \newcommand{\MaxOf}[1]{\max(#1)}
  \newcommand{\LipMaps}[1][]{\mathcal{F}_{#1}}
  
  \newcommand{\LipMapsOn}[2][]{\mathcal{F}_{#1}(#2)}
  \newcommand{\LevelLipConstOf}[2][]{|#2|_{\MetricParameter}^{#1}}
  \newcommand{\CardOf}[1]{\# #1}

  \newvariable{\MetricParameter}{\vartheta}
  \newvariable{\AltMetricParameter}{\psi}
  \newvariable{\TheLipConst}{C}
  \newvariable{\NumIterations}{m}
  \newvariable{\ConstOne}{\mathbf{1}}
  \newvariable{\TheEps}{\varepsilon}
  \newvariable{\IndexBound}{N}
  \newvariable{\TheLevel}{m}
  \newvariable{\AltLevel}{n}
  \newvariable{\MaxModDegree}{\bar{\ModDegree}}
  \newvariable{\TheCode}{w}
  \newvariable{\AltCode}{\bar{\TheCode}}
  \newvariable{\ThePrefix}{\TheCode}

  \newvariable{\TheRadius}{\rho}
  \newvariable{\FnDiff}{r}
  \newvariable{\PileHight}{R}
  \newvariable{\FiniteRank}{K}

  \newcommand{\Prefer}[1]{#1^{\bullet}}
  \newcommand{\ess}{\mathrm{ess}}
  \newvariable{\TheProjection}{\Pi}

  \newcommand{\union}{\cup}
  \newcommand{\ContDual}[1]{#1'}
  \newvariable{\TheEigenvalue}{\lambda}
  \newvariable{\AltEigenvalue}{\mu}
  \newvariable{\TheEigenspace}{E}
  \newvariable{\DualEigenspace}{E'}
  \newvariable{\AltEigenspace}{E'}
  \newvariable{\TheGenEigenspace}{K}
  \newvariable{\AltGenEigenspace}{L}
  \newvariable{\MoreIterations}{k}
  \newvariable{\TheBanachSpace}{X}
  \newvariable{\TheCore}{R}
  \newvariable{\TheCoreVector}{u}
  \newcommand{\directsum}[1][]{\oplus_{#1}}
  \newcommand{\DirectSum}[1][]{\oplus_{#1}}
  \newvariable{\Spectrum}{\sigma}
  \newvariable{\ComplexVar}{z}

  \newcommand{\LocConstOn}[1]{\mathcal{C}_{\mathrm{lc}}(#1)}
  \newcommand{\FinAddMeasuresOn}[1]{\mathcal{M}_{\mathrm{fa}}(#1)}
  \newcommand{\PairingOf}[2]{\langle #1\,|\, #2 \rangle}
  \newcommand{\diff}{\operatorname{d}}
  \newvariable{\AlgTf}{\Tf_{0}}
  \newvariable{\DualAlgTf}{\AlgTf'}
  \newvariable{\PrefixSet}{\mathbb{P}}
  \newvariable{\TheMeasure}{\mu}
  \newvariable{\CharFct}{\ConstOne}
  \newvariable{\TheEigenForm}{\varphi}
  \newvariable{\TheIsom}{\psi}
  \newvariable{\IslandInt}{\Psi}
  \newvariable{\TheVectorSpace}{V}
  \newvariable{\DualVectorSpace}{\TheVectorSpace'}
  \newvariable{\AltVectorSpace}{W}
  \newvariable{\TheLinOp}{A}
  \newvariable{\DualLinOp}{\TheLinOp'}
  \newvariable{\AltLinOp}{B}
  \newcommand{\Ker}{\operatorname{ker}}
  \newvariable{\Kernel}{\Ker}
  \newvariable{\TheNilpotency}{I}
  \newvariable{\DualNilpotency}{\TheNilpotency'}
  \newvariable{\AltNilpotency}{J}
  \newvariable{\TheComplement}{C}
  \newvariable{\DualComplement}{\TheComplement'}
  \newvariable{\AltComplement}{D}
  \newvariable{\TheVector}{v}
  \newvariable{\AltVector}{w}
  \newvariable{\DualVector}{\xi}
  \newcommand{\TheFormOf}[2]{(#1\,|\,#2)}

  \newcommand{\VertexDistance}[2]{\operatorname{dist}(#1,#2)}
  \newcommand{\TreeDistance}[1][]{\tilde{\Distance}_{#1}}
  \newcommand{\TreeDistanceOf}[3][]{\TreeDistance[#1](#2,#3)}
  \newcommand{\RayFromTo}[2]{#1\to#2}
  
  \newcommand{\HoroDistOf}[3][]{\langle #2, #3 \rangle_{#1}}
  \newcommand{\HoroBraket}[3][]{\langle #2, #3 \rangle_{#1}}
  
  \newcommand{\isomorphic}{\cong}
  \newcommand{\ForwardBoundaryOf}[1]{\TreeBoundary(#1)}
  \newcommand{\ShadowOf}[2][]{\TreeBoundary_{#1}(#2)}
  \newcommand{\CatZero}{{\small CAT(0)}}
  \newcommand{\intersect}{\cap}
  \newcommand{\Cocycle}[3][]{\zeta_{#1}^{#2,#3}}
  \newcommand{\FAMeasuresOn}[1]{\mathcal{M}_{\mathrm{fa}}(#1)}

  \newcommand{\PoissonTf}[1][]{\Psi_{#1}}
  \newcommand{\PoissonTfOf}[2][]{\PoissonTf[#1](#2)}
  \newcommand{\PoissonTfAt}[2][]{\Psi_{#1}^{#2}}
  \newcommand{\PoissonTfAtOf}[3][]{\PoissonTfAt[#1]{#2}(#3)}

  \newvariable{\TheTree}{\tilde{\TheGraph}}
  \newvariable{\TreeEdge}{\tilde{\TheEdge}}
  \newvariable{\TreeMeasure}{\tilde{\TheMeasure}}
  \newvariable{\TreeTurnSum}{\tilde{\TurnSum}}
  \newvariable{\TreeTf}{\tilde{\Tf}}
  \newvariable{\TreeDualTf}{\TreeTf'}
  \newvariable{\TreeNeighborSum}{\tilde{\NeighborSum}}
  \newvariable{\TreeNeighborAvg}{\tilde{\NeighborAvg}}
  \newvariable{\TreeGradient}{\tilde{\Gradient}}
  \newvariable{\TreeRescale}{\tilde{\Rescale}}
  \newvariable{\TreeLocalTweak}{\tilde{\LocalTweak}}
  \newvariable{\TreeBoundary}{\Omega}
  \newvariable{\TheEnd}{\xi}
  \newvariable{\AltEnd}{\chi}
  \newvariable{\TheRay}{\rho}
  \newvariable{\TreeBaseVertex}{\tilde{o}}
  \newvariable{\TheTreeVertex}{\tilde{\TheVertex}}
  \newvariable{\AltTreeVertex}{\tilde{\AltVertex}}
  \newvariable{\ThrTreeVertex}{\tilde{z}}
  \newvariable{\TheEigenFct}{\FnTreeVert}
  \newvariable{\TheBase}{K}
  \newvariable{\MultConst}{C}
  \newvariable{\TreeVertexSet}{\tilde{\VertexSet}}
  \newvariable{\TreeEdgeSet}{\tilde{\EdgeSet}}
  \newvariable{\TreePathSet}{\tilde{\PathSet}}
  \newvariable{\FnTreeVert}{\tilde{\FnVert}}
  \newvariable{\FnTreeEdge}{\tilde{\FnEdge}}
  \newvariable{\FnTreePath}{\tilde{\FnPath}}
  \newvariable{\FnEnds}{\FnTreePath}

  \newcommand{\Invariants}[2][\PiOne]{#2^{#1}}
  \newcommand{\rmod}{/}
  \newcommand{\OrbitSpace}[2]{#2 \rmod #1}
  \newcommand{\Aut}[1][]{\operatorname{Aut}_{#1}}
  \newcommand{\AutOf}[2][]{\Aut[#1](#2)}
  \newcommand{\Restriction}[1][]{\operatorname{res}_{#1}}
  \newcommand{\RestrictionOf}[2][]{\Restriction[#1](#2)}
  \newcommand{\TwistedInvariants}[2][\PiOne]{#2_{\TwistedAction}^{#1}}

  \newvariable{\PiOne}{\Gamma}
  \newvariable{\BaseVertex}{o}
  \newvariable{\TreeAutom}{\gamma}
  \newvariable{\InvAutom}{\TreeAutom^{-1}}
  \newvariable{\TwistedAction}{\rho^{\TheTwist}}
  \newvariable{\TreeCode}{\tilde{\TheCode}}
  \newvariable{\TreePath}{\tilde{\ThePath}}

\begin{document}

  \title{
    Spectral correspondences for finite graphs without dead ends
  }

  \author{K.-U.~Bux \and J.~Hilgert \and T.~Weich}

  \date{\PaperDate}
  
  \newcommand{\ContactInfo}{
    \par\noindent
    K.-U.~Bux,
    Fakult\"at f\"ur Mathematik,
    Universit\"at Bielefeld,
    Postfach 100131, D-33501 Bielefeld,
    Germany;
    \email{bux@math.uni-bielefeld.de}

    \medskip
    \par\noindent
    J.~Hilgert,
    Institut f\"ur Mathematik,
    Universit\"at Paderborn, D-33095 Paderborn,
    Germany;
    \email{hilgert@math.uni-paderborn.de}

    \medskip
    \par\noindent
    T.~Weich,
    Institut f\"ur Mathematik,
    Universit\"at Paderborn, D-33095 Paderborn,
    Germany;
    \email{weich@math.uni-paderborn.de}
  }
  
  \maketitle

  \begin{abstract}
    
    We compare the spectral properties of two kinds of linear
    operators characterizing the (classical) geodesic flow and its
    quantization on connected locally finite graphs without dead ends. The
    first kind are transfer operators acting on vector spaces associated
    with the set of non backtracking paths in the graphs. The second kind
    of operators are averaging operators acting on vector spaces
    associated with the space of vertices of the graph. The choice of
    vector spaces reflects regularity properties. Our main results are
    correspondences between classical and quantum spectral objects as well
    as some automatic regularity properties for eigenfunctions of transfer
    operators.
  \end{abstract}

  \renewcommand{\thefootnote}{}

  \footnotetext{All three authors acknowledge support by the Deutsche
    Forschungsgemeinschaft (DFG, German Research Foundation) via the grant
    SFB-TRR 358/1 2023 — 491392403}



  \section{Introduction}  
  
  Graphs have been used as simplified models in a wide variety of
  fields. In the process combinatorial objects are used to replace
  objects from the original context. This explains why there is a
  plethora of literature on graphs and combinatorial operators on
  function spaces defined in terms of the graph data. Often the
  interpretation of the simplified models in the spirit of the original
  models lead to interrelations between such operators.

  In this paper the relations we study are motivated by relations
  between geodesic flows and their quantizations as they were found in
  the context of compact hyperbolic spaces by Dyatlov, Faure and
  Guillarmou \cite{DFG15} and were later extended to other locally
  symmetric spaces \cite{GHW18a, Had20, KW20, KW21,GHW21,HHW21}.
  Moreover, Faure and Tsujii recently obtained generalizations of
  this quantum classical correspondence for geodesic flows on manifolds
  with variable negative curvature \cite{FT21}. However, the
  understanding of the occuring ``quantum'' operator is so far rather
  inexplicit.

  In a similar vein Anantharaman and collaborators used graphs as
  simplified models to study quantum ergodicity,
  \cite{ALM15,Ana17,AS19a,AS19b,AS19c}. In that context one also finds
  interrelations between the combinatorial operators considered,
  \cite{anantharaman2017some}.

  The study of combinatorial operators on graphs of course by far
  predates these recent efforts. Combinatorial Laplacians, which in our
  context model quantum dynamics, were considered already in the 19th
  century (according to \cite[\S~1.2]{CdV98} for the first time by
  Lagrange when discretizing the equations describing the propagation of
  sound; see also \cite{CDS95} and the references therein). Also, they
  show up prominently in the harmonic analysis of homogeneous trees
  which is not only a toy model for harmonic analysis on symmetric
  spaces, but also a way to do harmonic analysis on certain $p$-adic
  groups (starting with \cite{Ca72,Ca73}, see also \cite{FTN91}). Other
  applications can be found e.g. in data clustering, \cite{HHOS22}. The
  transfer operators, which we use to describe the shift dynamics on
  geodesic rays, can be described by the adjacency matrix of the graph,
  which has been studied intensively e.g. in the context of
  non-backtracking Markov chains (see e.g. \cite{LP16,LP16b}), but also in
  the context of network analysis (see e.g. \cite{SRB23}) and zeta
  functions (see e.g. \cite{Te11}). A relation between a slightly different Laplacian and transfer operator on triangulations of surfaces has recently been established and used in \cite{BCDS23}.

  It is clear from the above that bits and pieces of most
  combinatorial calculations one may do for graphs could be found
  somewhere in the literature. We nevertheless decided to give complete
  proofs for the benefit of the reader with the additional aim to make
  clear where then natural limits of our methods are.
  
  The guiding principle behind the spectral correspondences we want to
  describe is that there is a close relation between the classical
  motion, visualize a particle moving along the graph, and the generator
  of the corresponding quantum motion, which is modeled as a
  (combinatorial) Laplace operator. The easiest case in which this
  principle can be verified, are homogeneous trees where geodesics are
  bi-infinite paths, determined by a pair of distinct boundary
  points. The geodesic flow is given by the motion along such a
  geodesic, which is simply by shifting. On the other hand a natural
  combinatorial \emph{Laplacian} $\NeighborAvg$ acting on functions defined on
  the vertices of the tree is given by averaging the values of the
  function on neighboring points. It is clear what is meant by spectral
  theory of such a Laplacian: choose a space ${\LipMaps}$ of functions
  and determine the spectrum of $\NeighborAvg|_{\LipMaps}$. What is less clear
  is how to build a linear operator from the geodesic flow. There are
  various options, we choose to consider the geodesic flow on geodesic
  rays given by a starting point in the tree and a point in the
  boundary. Then the \emph{transfer operator} $\Tf$ acts on
  functions defined on such geodesic rays by adding up the values of the
  function on the geodesic rays which are shifted into the given one. It
  then turns out that the spectral theory of the transfer operator
  $\Tf$ is indeed very closely related to the Laplacian
  $\NeighborAvg$. The means of transportation between the two sides is the
  Poisson transform mapping generalized functions on the boundary to
  eigenfunctions of combinatorial Laplacians on the tree (see
  Corollary~\ref{cor:Poisson trafo}). These Poisson transforms are
  inspired by the Poisson transforms on Riemannian symmetric spaces and
  can be found for instance in \cite{FTN91}. A generalization to trees
  of bounded degree was worked out in \cite{BHW22}.

  If one wants to leave the world of trees and study the questions
  on finite graphs (which can be seen as quotients of trees), the
  combinatorial Laplacian is still naturally defined. In fact, its
  spectral theory is even simpler because there are only finitely many
  vertices.
  As far as geodesic flow on such finite graphs is concerned, recall
  that geodesic curves for Riemannian manifolds are by no means shortest
  connections between their endpoints. This is true only locally. For
  graphs this point of view says that in fact, any path in the graph
  which is not back-tracking should be considered a (local)
  geodesic. Locally, i.e. for a single edge, it will always be a
  shortest connection. But then the adjacency matrix shows how to extend
  geodesic rays by adding edges at the starting point, which yields a
  natural transfer operator.  Combinatorial manipulations which resemble
  calculations done in \cite{anantharaman2017some,BHW22} and other
  papers then yields linear bijections between equalizers of algebraic
  equations for transfer, respectively Laplace operators (see
  Propositions ~\ref{algebraic} and ~\ref{translated-alg-obs}). These
  equalizers reduce to eigenspaces when the graph is regular, i.e. each
  vertex has the same degree.

  In order to move on from algebraic equalizers to spectral theory of
  the transfer operator one has to restrict the operators to functions
  satisfying some regularity conditions. While
  \cite{anantharaman2017some} works with $\ell^2$-functions, we follow
  \cite{BHW22} in considering Lipschitz conditions as replacement of
  differentiability and thus obtain scales of generalized
  functions. This is more in the spirit of differential analysis on
  locally symmetric spaces. Together with a filtration of the spaces of
  functions on geodesic rays given by the number of edges on which the
  functions actually depend, this allows us to narrow down the location
  of elements in the algebraic equalizers belonging to our scale of
  spaces
  (see Theorem~\ref{main}, which can be read as an automatic
  regularity theorem and Corollary~\ref{transfer-on-loc-const} which is
  a formulation of the resulting spectral correspondence). In the case
  of regular graphs, Corollary~\ref{transfer-on-loc-const} gives an
  exact correspondence of the spectrum of the transfer operators of the
  geodesic flow and the spectrum of the combinatorial Laplacian on the
  regular graph. Such a correspondence is very much expected in view of
  the quantum-classical correspondence. Let us, however, emphasize that
  also in the case of nonregular finite graphs
  Corollary~\ref{transfer-on-loc-const} gives a completely explicit
  description of the spectrum of the transfer operator in terms of the
  combinatorial Laplacian and such an explicit description does not seem
  to be achievable at the moment for geodesic flows in variable
  curvature (see Remark~\ref{rem:correspondence} for further
  discussion).

  Furthermore, our approach allows us to give isomorphisms between
  eigen\-spaces of the transfer operator and eigenspaces (for the same
  eigenvalue) of its dual, which can be viewed as the graph analogue of
  resonant states for the geodesic flow on hyperbolic surfaces (see
  Theorem~\ref{dual-main}).
  
  There are at least two different methods to establish the spectral
  correspondences alluded to above. In this paper we use a mixture of
  combinatorial and functional analytic arguments working on spaces
  defined directly from the finite graph. An alternative is to work on
  the simply connected covering of the graph, which is a tree, and then
  study objects with invarince properties under the natural actions of
  the group of deck transformations. This approach allows one to use the
  additional tools and results available for trees such as the Poisson
  transforms mentioned earlier and it is in remarkable analogy to how
  quantum classical correspondence is obtained for locally symmetric
  spaces. In Section~\ref{sec:up-and-down} we explain this approach and
  complement the main results by adding spaces of invariants under the
  decktransformations to the spectral correspondences (see in particular
  Theorem~\ref{thm:spectral correspondence via covering}).

  For infinite graphs the methods we use in this paper do not suffice
  to go beyond algebraic equalizers in our quantum-classical spectral
  correspondences. On the quantum side, in view of the results for
  convex cocompact hyperbolic surfaces given in \cite{GHW18a}, one would
  expect resonances of the Laplacian to replace the pure point
  spectrum. To our knowledge such a theory of resonances for graphs has
  not yet been worked out. The technique using universal coverings and
  Poisson transforms is limited to regular parameters ensuring that the
  Poisson transforms are actually bijective. This limitation can
  actually be circumvented by using vector valued Poisson transforms
  just as in the case of rank one compact locally symmetric space (see
  \cite{AH23}). This issue will be addressed elsewhere.

  \section{Setting}
  Let $\TheGraph$ be a connected graph with vertex set
  $\VertexSet$. We do not allow loops or multiple edges. By $\EdgeSet$
  we denote the set of oriented edges of $\TheGraph$. The maps
  \(
    \Terminal\mapcolon\EdgeSet\rightarrow\VertexSet
  \)
  and
  \(
    \Initial\mapcolon\EdgeSet\rightarrow\VertexSet
  \)
  assign to each oriented edge $\TheEdge\in\EdgeSet$ its terminal
  vertex $\TerminalOf{\TheEdge}$ and its initial vertex
  $\InitialOf{\TheEdge}$, respectively. The map
  \(
    \Opposite\mapcolon\EdgeSet\rightarrow\EdgeSet
  \)
  reverses the orientation, i.e.,
  $\TerminalOf{\OppositeOf{\TheEdge}}=\InitialOf{\TheEdge}$ and
  $\InitialOf{\OppositeOf{\TheEdge}}=\TerminalOf{\TheEdge}$.

  We say that the vertices $\TheVertex$ and $\AltVertex$ are
  \notion{adjacent} or \notion{neighbors} if there is an edge $\TheEdge$
  with $\InitialOf{\TheEdge}=\TheVertex$ and
  $\TerminalOf{\TheEdge}=\AltVertex$.  In this case, we also say that
  $\TheEdge$ \notion{goes from $\TheVertex$ to $\AltVertex$}. Then, we
  write $\TheVertex \adj \AltVertex$. As we do not allow loops, a vertex
  does not neighbor itself. As we do not allow multiple edges, two
  adjacent vertices $\TheVertex$ and $\AltVertex$ are joined by exactly
  one edge from $\TheVertex$ to $\AltVertex$ (and the opposite edge from
  $\AltVertex$ to $\TheVertex$).

  For each vertex $\TheVertex\in\VertexSet$, let
  $\TheDegreeOf{\TheVertex}$ denote its \notion{degree}, i.e., the
  number of its neighbors. We assume that $\TheDegreeOf{\TheVertex}$ is
  always at least $2$.  That says that there are no isolated vertices
  (degree~$0$) and no terminal vertices (degree~$1$) in $\TheGraph$. In
  particular, this ensures that when we visit a vertex through an
  incoming edge, we can leave it without backtracking. In particular,
  any edge path can be continued ad infinitum.

  We also assume that the graph $\TheGraph$ is \notion{locally
    finite}, i.e., each vertex has finite degree. It will be convenient to
  consider the degree shifted down by $1$. So, we write
  $\TheDegreeOf{\TheVertex}=1+\ModDegreeOf{\TheVertex}$. 

  We say that two edges $\TheEdge$ and $\AltEdge$ form a \notion{turn}
  if $\TerminalOf{\TheEdge}=\InitialOf{\AltEdge}$ and
  $\TheEdge\neq\OppositeOf{\AltEdge}$.  In this case, we write
  $\TheEdge\turn\AltEdge$.  By $\PathSet$, we denote the set of
  \notion{infinite edge paths without backtracking}, i.e., sequences
  $\TheEdge[1]\turn\TheEdge[2]\turn\TheEdge[3]\turn\cdots$ of edges such
  that $\TheEdge[\TheIndex]$ and $\TheEdge[\TheIndex+1]$ form a turn for
  each $\TheIndex$. Note that the space of edge paths $\PathSet$ is an
  example of a shift space. If the set of oriented edges $\EdgeSet$ is
  finite then we are in the setting of subshifts of finite type (see
  e.g. \cite[Chapter 1]{Baladi}).
  
  For each index $\TheIndex$ we define the projection
  $\Projection[\TheIndex]\mapcolon\PathSet\rightarrow\EdgeSet$ by:
  \[
    \ProjectionOf[\TheIndex]{
      \TheEdge[1]\turn\TheEdge[2]\turn\TheEdge[3]\turn\cdots
    }
    :=
    \TheEdge[\TheIndex]
  \]
  Note that the absence of terminal vertices is equivalent to the
  surjectivity of each projection $\Projection[\TheIndex]$, whereas the
  absence of isolated vertices is equivalent to the surjectivity of the
  maps $\Initial\mapcolon\EdgeSet\rightarrow\VertexSet$ and
  $\Terminal\mapcolon\EdgeSet\rightarrow\VertexSet$.

  \section{An algebraic prelude}
  We start by considering function spaces without any imposition about
  regularity.  Let $\MapsFromTo{A}{B}$ denote the set of all maps from
  $A$ to $B$.  Then $\MapsFromTo{A}{\ComplexNumbers}$ is a
  $\ComplexNumbers$-vector space.

  The projections
  \(
    \PathSet\xrightarrow{\Projection[1]}
    \EdgeSet\xrightarrow{\Initial}
    \VertexSet
  \)
  induce inclusions
  \(
    \MapsOn{\VertexSet}\xrightarrow{\Inclusion}
    \MapsOn{\EdgeSet}\xrightarrow{\Inclusion}
    \MapsOn{\PathSet}
    .
  \)
  Under this interpretation, a function on $\PathSet$ belongs to the
  subspace $\MapsOn{\EdgeSet}$ if it only depends on the
  initial edge of the path. Similarly, a function defined on edges
  belongs to $\MapsOn{\VertexSet}$ if it only depends on
  the initial vertex of the edge.

  As the graph is locally finite, the \notion{neighbor sum operator}
  \begin{align*}
    \NeighborSum \mapcolon
    \MapsOn{\VertexSet}
    & \longrightarrow \MapsOn{\VertexSet}\\
    \FnVert
    & \longmapsto
      \BlockOf{
      \TheVertex \mapsto
      \Sum[\TheVertex\adj\AltVertex]{
      \FnVertOf{\AltVertex}
      }
      }
  \end{align*}
  is defined.  Similarly, summation over one-turn-extensions defines
  the \notion{transfer operator}:
  \begin{align*}
    \Tf\mapcolon
    \MapsOn{\PathSet}
    & \longrightarrow\MapsOn{\PathSet} \\
    \FnPath
    & \longmapsto
      \BlockOf{
      (\TheEdge[1]\turn\TheEdge[2]\turn\TheEdge[3]\turn\cdots)
      \mapsto
      \Sum[{\TheEdge[0]\turn\TheEdge[1]}]{
      \FnPathOf{\TheEdge[0]\turn\TheEdge[1]\turn\TheEdge[2]\turn\cdots}
      }
      }
  \end{align*}
  Note that this operator restricts to the subspace
  $\MapsOn{\EdgeSet}$. We call the restriction the
  \notion{turn sum operator} and use different notation:
  \begin{align*}
    \TurnSum \mapcolon
    \MapsOn{\EdgeSet}
    & \longrightarrow \MapsOn{\EdgeSet} \\
    \FnEdge
    & \longmapsto
      \BlockOf{
      \TheEdge
      \mapsto
      \Sum[\AltEdge\turn\TheEdge]{
      \FnEdgeOf{\AltEdge}
      }
      }
  \end{align*}

  Here, we want to relate the operators $\TurnSum$ and
  $\NeighborSum$. For each complex parameter $\TheTwist\neq 0$, we
  consider the \notion{twisted gradient} operator
  \begin{align*}
    \Gradient[\TheTwist]\mapcolon
    \MapsOn{\VertexSet}
    & \longrightarrow\MapsOn{\EdgeSet} \\
    \FnVert
    & \longmapsto
      (\FnVert\compose\Initial) - \TheTwist^{-1} (\FnVert\compose\Terminal)
  \end{align*}
  and compute the composition $\TurnSum\Gradient[\TheTwist]$. For a
  function $\FnVert\in\MapsOn{\VertexSet}$ and an edge
  $\TheEdge$, we find:
  \begin{align*}
    (\TurnSum\Gradient[\TheTwist]\,\FnVert)(\TheEdge)
    & = \Sum[\AltEdge\turn\TheEdge]{
      (\Gradient[\TheTwist]\,\FnVert)(\AltEdge)
      }
    \\
    & = \Sum[\AltEdge\turn\TheEdge]{
      \left(
      \FnVertOf{\InitialOf{\AltEdge}}
      -
      \TheTwist^{-1} \FnVertOf{\TerminalOf{\AltEdge}}
      \right)
      }
    \\
    & = (\NeighborSum\,\FnVert)(\InitialOf{\TheEdge})
      - \FnVertOf{\TerminalOf{\TheEdge}}
      - \TheTwist^{-1} \ModDegreeOf{\InitialOf{\TheEdge}} \FnVertOf{\InitialOf{\TheEdge}}
    \\
    & = (\NeighborSum\,\FnVert)(\InitialOf{\TheEdge})
      + \TheTwist(\Gradient[\TheTwist]\,\FnVert)(\TheEdge)
      - (\TheTwist+\TheTwist^{-1}\ModDegreeOf{\InitialOf{\TheEdge}})
      \FnVertOf{\InitialOf{\TheEdge}}
  \end{align*}
  This motivates the introduction of the \notion{twisted rescaling operator}:
  \begin{align*}
    \Rescale[\TheTwist] \mapcolon
    \MapsOn{\VertexSet}
    & \longrightarrow \MapsOn{\VertexSet} \\
    \FnVert
    & \longmapsto
      \BlockOf{
      \TheVertex \mapsto (\TheTwist+\ModDegreeOf{\TheVertex}\TheTwist^{-1})
      \FnVertOf{\TheVertex}
      }
  \end{align*}
  which is simply a multiplication operator with the function
  \(
    \TheVertex\mapsto\TheTwist + \ModDegreeOf{\TheVertex}\TheTwist^{-1}.
  \)
  Note that in the case of constant degree
  $\ModDegreeOf{\TheVertex}=\MaxModDegree$ the twisted rescaling
  operator is simply the multiplication with a scalar.
  
  With this, we can write our finding as follows:
  \begin{equation}\label{OperatorIdentity}
    (\TurnSum-\TheTwist\Identity)\Gradient[\TheTwist]
    =
    \Inclusion(\NeighborSum - \Rescale[\TheTwist])
  \end{equation}

  \begin{proposition}\label{algebraic}
    For $\TheTwist\not\in\SetOf{-1,0,1}$, the twisted gradient
    $\Gradient[\TheTwist]$ defines an isomorphism from the equalizer
    \(
      \Equalizer{\NeighborSum}{\Rescale[\TheTwist]}
      \subspace\MapsOn{\VertexSet}
    \)
    to the equalizer
    \(
      \Equalizer{\TurnSum}{\TheTwist\Identity}
      \subspace\MapsOn{\EdgeSet}
    \),
    i.e., the eigenspace of $\TurnSum$ to the eigenvalue $\TheTwist$.

    The inverse is given as the restriction of
    $(\TheTwist-\TheTwist^{-1})^{-1}\InSum$ where
    \begin{align*}
      \InSum \mapcolon
      \MapsOn{\EdgeSet}
      & \longrightarrow\MapsOn{\VertexSet} \\
      \FnEdge
      & \longmapsto
        \BlockOf{
        \TheVertex \mapsto
        \Sum[\TerminalOf{\TheEdge}=\TheVertex]{
        \FnEdgeOf{\TheEdge}
        }
        }
    \end{align*}
    is given by summing over incoming edges.
  \end{proposition}
  \begin{proof}
    First note that $\FnVert\in\MapsOn{\VertexSet}$ lies in the kernel
    of $\Gradient[\TheTwist]$ if and only if
    $\TheTwist\FnVertOf{\InitialOf{\TheEdge}}=\FnVertOf{\TerminalOf{\TheEdge}}$
    for each edge $\TheEdge$. By considering $\TheEdge$ and
    $\OppositeOf{\TheEdge}$, we find that $\TheTwist^2\FnVert=\FnVert$ is
    a necessary condition for being in the kernel.  In particular,
    $\Gradient[\TheTwist]$ is injective for
    $\TheTwist\not\in\SetOf{-1,0,-1}$.
    
    Also, observe the following identities:
    \begin{align*}
      (\InSum\, \FnEdge)( \InitialOf{\TheEdge} )
      & = (\TurnSum\, \FnEdge)(\TheEdge) + \FnEdgeOf{\OppositeOf{\TheEdge}} \\
      (\InSum\, \FnEdge)( \TerminalOf{\TheEdge} )
      & = (\TurnSum\, \FnEdge)( \OppositeOf{\TheEdge} ) + \FnEdgeOf{ \TheEdge }
    \end{align*}
    For $\FnEdge\in\Equalizer{\TurnSum}{\TheTwist\Identity}$, we find:
    \begin{align*}
      (\Gradient[\TheTwist]\InSum\,\FnEdge)(\TheEdge)
      & = (\InSum\,\FnEdge)(\InitialOf{\TheEdge})
        - \TheTwist^{-1}(\InSum\,\FnEdge)(\TerminalOf{\TheEdge})
      \\
      & = \TheTwist\FnEdgeOf{\TheEdge} + \FnEdgeOf{\OppositeOf{\TheEdge}}
        - \TheTwist^{-1}\TheTwist \FnEdgeOf{\OppositeOf{\TheEdge}}
        - \TheTwist^{-1} \FnEdgeOf{\TheEdge}
      \\
      & = (\TheTwist-\TheTwist^{-1}) \FnEdgeOf{\TheEdge}
    \end{align*}
    It follows that $\Gradient[\TheTwist]$ is also onto, whence it is
    an isomorphism whose inverse is given by
    $(\TheTwist-\TheTwist^{-1})^{-1}\InSum$.
  \end{proof}

  An alternative perspective, more in line with analogies to analysis
  on manifolds, is to consider the Laplacian $\NeighborAvg$, i.e.,
  averaging over neighboring vertices instead of summing over them.
  We adjust the twisted rescaling operator accordingly.
  Specifically, we define:
  \begin{align*}
    \NeighborAvg\mapcolon
    \MapsOn{\VertexSet}
    & \longrightarrow \MapsOn{\VertexSet}
    \\
    \FnVert & \longmapsto
              \BlockOf{
              \TheVertex \mapsto
              \frac{1}{1+\ModDegreeOf{\TheVertex}}
              \Sum[\AltVertex\adj\TheVertex]{
              \FnVertOf{\AltVertex}
              }
              }
    \\
    \LocalTweak[\TheTwist]\mapcolon
    \MapsOn{\VertexSet}
    & \longrightarrow \MapsOn{\VertexSet}
    \\
    \FnVert & \longmapsto
              \BlockOf{
              \TheVertex \mapsto
              \frac{
              \TheTwist+\TheTwist^{-1}\ModDegreeOf{\TheVertex}
              }{
              1+\ModDegreeOf{\TheVertex}
              }
              \FnVertOf{\TheVertex}
              }
  \end{align*}
  Note that the equalizer
  \(
    \Equalizer{\NeighborAvg}{\LocalTweak[\TheTwist]}
  \)
  equals the equalizer
  \(
    \Equalizer{\NeighborSum}{\Rescale[\TheTwist]}
  \).
  Hence, Proposition~\ref{algebraic} translates into this setting as
  follows:
  \begin{proposition}\label{translated-alg-obs}
    For $\TheTwist\not\in\SetOf{-1,0-1}$, the twisted gradient
    $\Gradient[\TheTwist]$ defines an isomorphism from the equalizer
    \(
      \Equalizer{\NeighborAvg}{\LocalTweak[\TheTwist]}
    \)
    to the eigenspace
    \(
      \Equalizer{\TurnSum}{\TheTwist\Identity}
    \)
    of $\TurnSum$ to the eigenvalue $\TheTwist$.
  \end{proposition}
  \begin{remark}
    Correspondences between the spectra of $\NeighborAvg, \Sigma$ and
    $S$ as stated in Proposition~\ref{algebraic} and
    \ref{translated-alg-obs} were studied in several places in the
    literature: for example in \cite{Smi07} equations~(44) and~(45), in
    \cite[Section 7]{ALM15} and in \cite[Section 3.1]{LP16}. In the two
    last mentioned articles the authors even provide a more detailed
    description of this relation including the exceptional parameters
    $z=\pm1$ and a description of Jordan blocks. However they only study
    the case of regular graphs. The only reference we are aware of that
    also treats non-regular graphs and that establishes a relation between
    eigenfunctions of $\TurnSum$ and elements in the equalizer
    $\Equalizer{\NeighborAvg}{\LocalTweak[\TheTwist]}$ is \cite{anantharaman2017some};
    see equation~(4) and the following discussion therein.
  \end{remark}

  \section{Bounded functions}
  As our first assumption on regularity, we discuss bounded functions
  on $\VertexSet$, $\EdgeSet$, and $\PathSet$. Let the corresponding
  function spaces be denoted by $\BoundedMaps{\VertexSet}$,
  $\BoundedMaps{\EdgeSet}$, and $\BoundedMaps{\PathSet}$, respectively.
  We endow these spaces with the sup-norm $\NormOf[\infty]{\cdot}$,
  turning them into Banach spaces.  The $\NormOf[\infty]{\cdot}$-limit
  of a Cauchy sequence of functions is given by taking limits pointwise.

  We define a filtration of $\BoundedMaps{\PathSet}$ as
  follows. Let $\Depends[\TheLength]$ be the subspace of
  $\BoundedMaps{\PathSet}$ consisting of those $\FnPath$ where
  $\FnPathOf{\TheEdge[1]\turn\TheEdge[2]\turn\cdots}$ only depends on
  $\TheEdge[1],\TheEdge[2],\ldots,\TheEdge[\TheLength]$.
  \begin{lemma}
    $\Depends[\TheLength]$ is a closed subspace of $\BoundedMaps{\PathSet}$.
  \end{lemma}
  \begin{proof}
    Let $(\FnPath[\TheIndex])$ be a Cauchy sequence in the subspace
    $\Depends[\TheLength]$ and let $\FnPath[\infty]$ be its limit in
    $\BoundedMaps{\PathSet}$. We have to show that $\FnPath[\infty]$
    lies in $\Depends[\TheLength]$, i.e., that
    \(
      \FnPathOf[\infty]{\ThePath}
      =
      \FnPathOf[\infty]{\AltPath}
    \)
    whenever $\ThePath$ and $\AltPath$ share the first $\TheLength$ edges.
    However, in this case, 
    \(
      \FnPathOf[\TheIndex]{\ThePath}
      =
      \FnPathOf[\TheIndex]{\AltPath}
    \)
    for each $\TheIndex$ and equality also holds in the limit.
  \end{proof}

  The inclusion operators
  $\Inclusion\mapcolon\MapsOn{\VertexSet}\rightarrow\MapsOn{\EdgeSet}$
  and $\Inclusion\mapcolon\MapsOn{\EdgeSet}\rightarrow\MapsOn{\PathSet}$
  when restricted to $\BoundedMaps{\VertexSet}$ and
  $\BoundedMaps{\EdgeSet}$, respectively, are bounded with operator
  norm~$1$. Note that $\Inclusion$ identifies $\BoundedMaps{\EdgeSet}$
  with $\Depends[1]$.

  The twisted gradient $\Gradient[\TheTwist]$ has operator norm at
  most $1+\AbsValueOf{\TheTwist}$.  The twisted rescaling operator
  $\Rescale[\TheTwist]$ as well as the summation operators
  $\NeighborSum$ and $\TurnSum$ are only bounded if there is a uniform
  bound on the degree of vertices in the graph $\TheGraph$.  For this
  reason, we will be especially interested in the case that there is a
  uniform upper bound $\MaxModDegree+1$ for the degree of vertices of the
  graph $\TheGraph$.

  \section{Lipschitz continuous functions}
  Our aim will be to study the an appropriate spectral theory of
  resonances for the transfer operator $\Tf$ in the spirit of
  \cite[\S~1.3]{Baladi}. In order to make this intrinsic discrete
  spectrum, often called Ruelle resonances, appear we need
  to let the transfer operator act on function spaces with a certain
  regularity. A standard choice is the scale of Banach spaces of
  Lipschitz continuous functions that we will introduce in this
  section.
  
  Fix a parameter $\MetricParameter\in(0,1)$ and define a metric (in
  fact, an ultrametric) $\Distance[\MetricParameter]$ on $\PathSet$ by:
  \[
    \DistanceOf[\MetricParameter]{
      \TheEdge[0]\turn\TheEdge[1]\turn\cdots
    }{
      \AltEdge[0]\turn\AltEdge[1]\turn\cdots
    }
    :=
    \MetricParameter^{\inf\SetOf[\TheIndex]{\TheEdge[\TheIndex]\neq\AltEdge[\TheIndex]}}
  \]
  The metric defines a topology on $\PathSet$, which does not depend
  on $\MetricParameter$.  In fact, the basic open sets admit a rather
  combinatorial description as follows.  For each vertex
  $\TheVertex\in\VertexSet$, we shall refer to the subset
  $\PathSet[\TheVertex]$ of edge paths starting at the vertex
  $\TheVertex$ as the \notion{island} designated by $\TheVertex$. Each
  edge path
  \(
    \ThePrefix = \TheEdge[1]\turn\cdots\turn\TheEdge[\TheLevel]
  \)
  of length $\TheLevel$ defines a \notion{district of level}
  $\TheLevel$ in the island designated by
  $\InitialOf{\TheEdge[1]}$. This district is the subset
  $\PathSet[\ThePrefix]$ of those infinite paths that start with the
  prefix $\ThePrefix$.  We think of $\ThePrefix$ as a \notion{postal
    code} with $\TheLevel$ digits that designates its district. The
  topology on $\PathSet$ is defined by regarding districts as basic open
  sets. Since the graph $\TheGraph$ is locally finite, each island is a
  totally disconnected compact set. In fact, one may think of points in
  $\PathSet$ as ends of a forest where each vertex
  $\TheVertex\in\VertexSet$ is a root and each $\TheLevel$-digit postal
  code is a vertex at distance $\TheLevel$ from its root. The children
  of a postal code $\ThePrefix$ are the one-digit-extensions
  $\ThePrefix\turn\AltEdge$. In this language, the space
  $\Depends[\TheLevel]$ consists of those bounded functions that are
  constant on districts of level $\TheLevel$ (and higher).

  Points on different islands, somewhat artificially, are at
  distance~$1$ in $\Distance[\MetricParameter]$. Consequently, any
  Lipschitz continuous function on $\PathSet$ is bounded. On the other
  hand, even a bounded function that is constant on islands may have a
  very large Lipschitz constant. We would like to think of the optimal
  Lipschitz constant for a function as akin to the sup-norm of its
  derivative, i.e., we would like Lipschitz constants to capture a more
  local behavior.

  We call a function
  $\FnPath\mapcolon\PathSet\rightarrow\ComplexNumbers$ \notion{uniformly
    $\Distance[\MetricParameter]$-Lipschitz continuous on islands} if
  there is a constant $\TheLipConst$ such that the restriction of
  $\FnPath$ to each island is $\Distance[\MetricParameter]$-Lipschitz
  continuous with Lipschitz constant $\TheLipConst$. Equivalently, one
  can require that
  \[
    \AbsValueOf{
      \FnPathOf{\ThePath}-\FnPathOf{\AltPath}
    }
    \leq
    \TheLipConst\DistanceOf[\MetricParameter]{\ThePath}{\AltPath}
  \]
  holds whenever $\ThePath$ and $\AltPath$ issue from the same
  vertex. Points in different level-$\TheLevel$ districts have at least
  distance $\MetricParameter^{\TheLevel-1}$. Consequently, we observe the
  following:
  \begin{observation}
    For a bounded function
    \(
      \FnPath\mapcolon\PathSet\rightarrow\ComplexNumbers
    \), the following are equivalent:
    \begin{enumerate}
      \item
        The map $\FnPath$ is $\Distance[\MetricParameter]$-Lipschitz
        continuous.
      \item
        The map $\FnPath$ is uniformly
        $\Distance[\MetricParameter]$-Lipschitz continuous on islands.
      \item
        The map $\FnPath$ is uniformly
        $\Distance[\MetricParameter]$-Lipschitz continuous on level-$1$
        districts.
      \item[$\vdots$]
      \item[m.]The map $\FnPath$ is uniformly
        $\Distance[\MetricParameter]$-Lipschitz continuous on level-$(m-2)$
        districts.
      \item[$\vdots$]
    \end{enumerate}
  \end{observation}

  Let $\LipMaps[\MetricParameter]$ denote the subspace of
  $\BoundedMaps{\PathSet}$ that consists of maps
  $\FnPath\mapcolon\PathSet\rightarrow\ComplexNumbers$ that are
  uniformly $\Distance[\MetricParameter]$-Lipschitz continuous on
  islands.
  As $\Depends[\TheLevel]$ consists of those functions that are
  constant on level-$\TheLevel$ districts, the previous observation
  immediately implies the following:
  \begin{lemma}
    $\Depends[\TheLevel]$ is a subspace of
    $\LipMaps[\MetricParameter]$ for any $\MetricParameter$, i.e., a
    function $\FnPath$ in $\Depends[\TheLevel]$ is Lipschitz continuous
    with respect to $\Distance[\MetricParameter]$.
  \end{lemma}

  Let $\LevelLipConstOf[0]{\FnPath}$ be the infimum over all such
  $\TheLipConst>0$ for which
  \[
    \AbsValueOf{
      \FnPathOf{\ThePath}-\FnPathOf{\AltPath}
    }
    \leq
    \TheLipConst\DistanceOf[\MetricParameter]{\ThePath}{\AltPath}
  \]
  holds whenever $\ThePath$ and $\AltPath$ lie in the same island.
  I.e., $\LevelLipConstOf[0]{\FnPath}$ is an optimal uniform
  Lipschitz constant on islands. Similarly, let
  $\LevelLipConstOf[1]{\FnPath}$ be the optimal Lipschitz constant on
  level-$1$ districts; and more generally, let
  $\LevelLipConstOf[\TheLevel]{\FnPath}$ be the optimal Lipschitz
  constant on level-$\TheLevel$ districts.  Clearly, those optimal
  Lipschitz constants are subadditive, whence each
  $\LevelLipConstOf[\TheLevel]{\DummyArg}$ is a semi-norm.

  We define a norm on $\LipMaps[\MetricParameter]$ by:
  \[
    \InnerNormOf[\MetricParameter]{\FnPath}
    :=
    \LevelLipConstOf[1]{\FnPath}
    +
    \NormOf[\infty]{\FnPath}
  \]
  We chose the optimal Lipschitz constant on level-$1$ districts for
  technical reasons that shall become clear later.

  \begin{remark}
    A more conventional norm combines the sup-norm with a global
    Lipschitz constant.  Let $\InfLipConstOf[\MetricParameter]{\FnPath}$ be
    the infimum over all (global) Lipschitz constants for $\FnPath$, and
    put:
    \[
      \LipNormOf[\MetricParameter]{\FnPath}
      :=
      \InfLipConstOf[\MetricParameter]{\FnPath}
      +
      \NormOf[\infty]{\FnPath}
    \]

    The norms $\LipNormOf[\MetricParameter]{\DummyArg}$ and
    $\InnerNormOf[\MetricParameter]{\DummyArg}$ are easily seen to be
    equivalent:
    \[
      \InnerNormOf[\MetricParameter]{\FnPath}
      \leq
      \LipNormOf[\MetricParameter]{\FnPath}
      \leq
      3\InnerNormOf[\MetricParameter]{\FnPath}
    \]
    Here, we use the estimate
    \(
      \InfLipConstOf[\MetricParameter]{\FnPath}
      \leq
      \MaxOf{
        \LevelLipConstOf[\TheLevel]{\FnPath}
        ,
        2
        \MetricParameter^{1-\TheLevel}
        \NormOf[\infty]{\FnPath}
      }
    \)
    which again follows from the minimum distance
    $\MetricParameter^{\TheLevel-1}$ of points in different
    level-$\TheLevel$ districts.
  \end{remark}

  \begin{proposition}
    $\LipMaps[\MetricParameter]$ is a Banach space with respect to the
    norm $\InnerNormOf[\MetricParameter]{\DummyArg}$.  Moreover, each
    $\InnerNormOf[\MetricParameter]{\DummyArg}$-Cauchy sequence is a
    $\NormOf[\infty]{\DummyArg}$-Cauchy sequence and its
    $\NormOf[\infty]{\DummyArg}$-limit is its
    $\InnerNormOf[\MetricParameter]{\DummyArg}$-limit. In particular, the
    limit function is given by taking limits pointwise.
  \end{proposition}
  \begin{proof}
    Clearly,
    \(
      \NormOf[\infty]{\FnPath}
      \leq
      \InnerNormOf[\MetricParameter]{\FnPath}
    \)
    for any $\FnPath\in\LipMaps[\MetricParameter]$. Hence a
    $\InnerNormOf[\MetricParameter]{\DummyArg}$-Cauchy sequence
    $(\FnPath[\TheIndex])$ is also a $\NormOf[\infty]{\DummyArg}$-Cauchy
    sequence and converges in $\BoundedMaps{\PathSet}$. Let
    $\FnPath[\infty]$ be the $\NormOf[\infty]{\DummyArg}$-limit of
    $(\FnPath[\TheIndex])$.

    Since $(\FnPath[\TheIndex])$ is a Cauchy sequence, there is a
    uniform upper bound on
    $\InnerNormOf[\MetricParameter]{\FnPath[\TheIndex]}$ and therefore on
    $\InnerConstOf[\MetricParameter]{\FnPath[\TheIndex]}$. It follows that
    the pointwise limit $\FnPath[\infty]$ is uniformly Lipschitz
    continuous on top-level districts, whence
    $\FnPath[\infty]\in\LipMaps[\MetricParameter]$.

    It remains to show that
    $\InnerConstOf[\MetricParameter]{\FnPath[\TheIndex]-\FnPath[\infty]}$
    tends to $0$ as $\TheIndex$ goes to infinity. Note that for two edge
    paths $\ThePath$ and $\AltPath$ in the same top-level district, we
    have
    \[
      \AbsValueOf{
        (\FnPath[\infty]-\FnPath[\TheIndex])(\ThePath)
        -
        (\FnPath[\infty]-\FnPath[\TheIndex])(\AltPath)
      }
      \leq
      2\NormOf[\infty]{\FnPath[\infty]-\FnPath[\AltIndex]}
      +
      \InnerConstOf[\MetricParameter]{
        \FnPath[\AltIndex]-\FnPath[\TheIndex]
      }
      \DistanceOf[\MetricParameter]{\ThePath}{\AltPath}
    \]
    For $\TheEps>0$ there is an $\IndexBound$ so that
    \(
      \InnerConstOf[\MetricParameter]{\FnPath[\AltIndex]-\FnPath[\TheIndex] }
      <
      \TheEps
    \)
    whenever $\TheIndex,\AltIndex>\IndexBound$.  Since
    $\NormOf[\infty]{\FnPath[\infty]-\FnPath[\AltIndex]}$ tends to $0$ as
    $\AltIndex$ goes to infinity and
    \(
      \InnerConstOf[\MetricParameter]{
        \FnPath[\AltIndex]-\FnPath[\TheIndex]
      }
      \DistanceOf[\MetricParameter]{\ThePath}{\AltPath}
    \)
    stays bounded by
    \(
      \TheEps\DistanceOf[\MetricParameter]{\ThePath}{\AltPath}
    \), it follows that
    \(
      \InnerConstOf[\MetricParameter]{\FnPath[\TheIndex]-\FnPath[\infty]}
      \leq
      \TheEps
    \)
    for $\TheIndex$ sufficiently large.
  \end{proof}

  \begin{corollary}
    For any level $\TheLevel\in \NatNumbers$, the function space
    $\Depends[\TheLevel]$ is a closed $\Tf$ invariant subspace of
    $\LipMaps[\MetricParameter]$.
  \end{corollary}
  \begin{proof}
    A $\InnerNormOf[\MetricParameter]{\DummyArg}$-Cauchy sequence in
    $\Depends[\TheLevel]$ is a $\NormOf[\infty]{\DummyArg}$-Cauchy
    sequence. As $\Depends[\TheLevel]$ is closed in
    $\BoundedMaps{\PathSet}$, its $\LipNormOf[\infty]{\DummyArg}$-limit lies
    in $\Depends[\TheLevel]$ and it also agrees with the
    $\InnerNormOf[\MetricParameter]{\DummyArg}$-limit in
    $\LipMaps[\MetricParameter]$.
    
    The $\Tf$ invariance is directly seen by its definition 
    \[
      (\Tf\FnPath)(\TheEdge[1]\turn\TheEdge[2]\turn\cdots)
      =
      \Sum[{\TheEdge[0]\turn\TheEdge[1]}]{
        \FnPathOf{\TheEdge[0]\turn\TheEdge[1]\turn\TheEdge[2]\turn\cdots}
      }
    \]
    which even yields that $\Tf$ maps
    \(
      \Depends[\TheLevel]\to\Depends[\TheLevel-1]
    \)
    provided that $\TheLevel>1$.
  \end{proof}

  \begin{observation}
    A function $\FnPath\in\Depends[1]$ is constant on level-$1$
    districts. Hence
    $\InnerConstOf[\MetricParameter]{\FnPath}=0$. Therefore, the norms
    $\NormOf[\infty]{\DummyArg}$ and
    $\InnerNormOf[\MetricParameter]{\DummyArg}$ agree on the subspace
    $\Depends[1]$.
  \end{observation}
  This is the main technical advantage in choosing
  $\LevelLipConstOf[1]{\DummyArg}$ as our preferred semi-norm. As the
  space $\Depends[1]$ is the space of bounded functions on edges, the
  restriction of the transfer operator to $\Depends[1]$ is just reduced
  to the turn sum operator.
  \[
    (\Tf\FnEdge)(\TheEdge)
    =
    \Sum[{\AltEdge\turn\TheEdge}]{
      \FnEdge(\AltEdge)
    }.   
  \]
  \begin{lemma}\label{rescaling-bounds}
    Assume that the degree of vertices in $\TheGraph$ is bounded from
    above, i.e., $\ModDegreeOf{\TheVertex}\leq\MaxModDegree$ for each
    vertex $\TheVertex$ of $\TheGraph$.  For any
    $\FnPath\in\LipMaps[\MetricParameter]$, we have
    \begin{align*}
      \NormOf[\infty]{\Tf\FnPath}
      & \leq \MaxModDegree\NormOf[\infty]{\FnPath}
      \\
      \InnerConstOf[\MetricParameter]{ \Tf \FnPath }
      & \leq \MetricParameter \MaxModDegree
        \InnerConstOf[\MetricParameter]{\FnPath}
    \end{align*}
  \end{lemma}
  \begin{proof}
    For the first inequality, recall
    \[
      (\Tf\FnPath)(\TheEdge[1]\turn\TheEdge[2]\turn\cdots)
      =
      \Sum[{\TheEdge[0]\turn\TheEdge[1]}]{
        \FnPathOf{\TheEdge[0]\turn\TheEdge[1]\turn\TheEdge[2]\turn\cdots}
      }
    \]
    There are at most $\MaxModDegree$ summands, each of which is
    bounded in absolute value by $\NormOf[\infty]{\FnPath}$.
    
    Consider two paths $\ThePath$ and $\AltPath$ on the same top-level
    district, say the one with postal code $\TheEdge$. Then, we find:
    \begin{align*}
      \AbsValueOf{
      (\Tf\FnPath)(\ThePath)
      -
      (\Tf\FnPath)(\AltPath)
      }
      & \leq
        \Sum[\AltEdge\turn\TheEdge]{
        \AbsValueOf{
        \FnPathOf{\AltEdge\turn\ThePath}
        -
        \FnPathOf{\AltEdge\turn\AltPath}
        }
        }
      \\
      & \leq
        \Sum[\AltEdge\turn\TheEdge]{
        \InnerConstOf[\MetricParameter]{\FnPath}
        \MetricParameter
        \DistanceOf[\MetricParameter]{\ThePath}{\AltPath}
        }
      \\
      & \leq
        \MaxModDegree
        \MetricParameter
        \InnerConstOf[\MetricParameter]{\FnPath}
        \DistanceOf[\MetricParameter]{\ThePath}{\AltPath}
    \end{align*}
    The claim follows.
  \end{proof}

  \begin{corollary}
    Assume that the degree of vertices in $\TheGraph$ is bounded from
    above.  Then the transfer operator $\Tf$ restricts to a bounded
    endomorphism on $\LipMaps[\MetricParameter]$. In fact,
    \(
      \MaxModDegree
      :=
      \Max{
        \SetOf[\ModDegreeOf{\TheVertex}]{
          \TheVertex\in\VertexSet
        }
      }
    \)
    is an upper bound for the operator norm of $\Tf$.
  \end{corollary}
  \begin{proof}
    Consider a function $\FnPath\in\LipMaps[\MetricParameter]$. We have
    \begin{align*}
      \InnerNormOf[\MetricParameter]{\Tf \FnPath}
      & = \InnerConstOf[\MetricParameter]{\Tf \FnPath}
        + \NormOf[\infty]{\Tf \FnPath}
      \\
      & \leq
        \MetricParameter\MaxModDegree
        \InnerConstOf[\MetricParameter]{\FnPath}
        +
        \MaxModDegree
        \NormOf[\infty]{\FnPath},
      \\
      & \leq
        \MaxModDegree
        (
        \InnerConstOf[\MetricParameter]{\FnPath}
        +
        \NormOf[\infty]{\FnPath}
        )
      \\
      & = \MaxModDegree\InnerNormOf[\MetricParameter]{\FnPath}
    \end{align*}
    The claim follows.    
  \end{proof}

  \section{The spectral radius}
  We continue under the standing hypothesis that the vertex degree is
  uniformly bounded in $\TheGraph$, and we put
  \(
    \MaxModDegree :=
    \Max{
      \SetOf[\ModDegreeOf{\TheVertex}]{
        \TheVertex\in\VertexSet
      }
    }
  \).
  Then $\Tf$ is a bounded operator with operator norm not exceeding
  $\MaxModDegree$. We want to describe the spectral theory of $\Tf$ and
  follow the well established ideas for subshifts of finite type that
  are for example described in \cite[Chapter~1]{Baladi}.

  It is easy to write down the explicit formula for powers
  $\Tf^\NumIterations$. We find:
  \[
    (\Tf^\NumIterations \FnPath)(\TheEdge[0]\turn\TheEdge[1]\turn\cdots)
    =
    \Sum[{
      \TheEdge[-\NumIterations]\turn\cdots\turn\TheEdge[-1]\turn\TheEdge[0]
    }]{
      \FnPathOf{\TheEdge[-\NumIterations]\turn\cdots\turn\TheEdge[-1]\turn
        \TheEdge[0]\turn\TheEdge[1]\turn\cdots}
    }
  \]
  If we apply $\Tf^\NumIterations$ to the constant function
  $\ConstOne$, we solve a counting problem for paths of length
  $\NumIterations$ that turn into a given edge $\TheEdge[0]$. We find:
  \[
    (\Tf^\NumIterations \ConstOne)(\TheEdge[0]\turn\TheEdge[1]\turn\cdots)
    =
    \CardOf{
      \SetOf{
        \TheEdge[-\NumIterations]\turn\cdots\turn\TheEdge[-1]\turn\TheEdge[0]
      }
    }
    \leq
    \MaxModDegree^\NumIterations
  \]
  We define:
  \[
    \PileHightOf[\NumIterations]{\TheEdge[0]} :=
    \CardOf{
      \SetOf{
        \TheEdge[-\NumIterations]\turn\cdots\turn\TheEdge[-1]\turn\TheEdge[0]
      }      
    }
  \]
  and:
  \[
    \PileHight[\NumIterations] :=
    \sup\SetOf[{\PileHightOf[\NumIterations]{\TheEdge[0]}}]{
      \TheEdge[0]\in\EdgeSet
    }
  \]
  The constant function $\ConstOne$ lies in $\Depends[1]$, and so do all
  iterates $\Tf^{\NumIterations}\ConstOne$. Thus, we find
  \(
    \PileHight[\NumIterations] =
    \NormOf[\infty]{ \Tf^{\NumIterations} \ConstOne }
    =
    \InnerNormOf[\MetricParameter]{ \Tf^{\NumIterations} \ConstOne }
    ,
  \)
  whence $\PileHight[\NumIterations]$ is a lower bound for the
  operator norm of $\Tf^{\NumIterations}$; and it does not even
  matter whether we consider the norm $\NormOf[\infty]{\cdot}$ or
  $\InnerNormOf[\MetricParameter]{\cdot}$ on
  $\LipMaps[\MetricParameter]$.

  On the other hand, by the same arguments as in
  Lemma~\ref{rescaling-bounds}, we find that for arbitrary
  $\FnPath\in\LipMaps[\MetricParameter]$, the following holds:
  \begin{align*}
    \NormOf[\infty]{\Tf^{\NumIterations} \FnPath}
    & \leq \PileHight[\NumIterations] \NormOf[\infty]{\FnPath}
    \\
    \InnerConstOf[\MetricParameter]{\Tf^{\NumIterations} \FnPath}
    & \leq
      \MetricParameter^{\NumIterations}
      \PileHight[\NumIterations] \InnerConstOf[\MetricParameter]{\FnPath}
    \\
    \InnerNormOf[\MetricParameter]{\Tf^{\NumIterations} \FnPath}
    & \leq
      \MetricParameter^{\NumIterations}
      \PileHight[\NumIterations] \InnerConstOf[\MetricParameter]{\FnPath}
      +
      \PileHight[\NumIterations] \NormOf[\infty]{\FnPath}
      =
      \PileHight[\NumIterations](
      \MetricParameter^{\NumIterations}
      \InnerConstOf[\MetricParameter]{\FnPath}
      +
      \NormOf[\infty]{\FnPath}
      )
      \leq
      \PileHight[\NumIterations] \InnerNormOf[\MetricParameter]{\FnPath}
  \end{align*}
  We deduce the following:
  \begin{observation}
    The operator norm of $\Tf^{\NumIterations}$ on
    $\LipMaps[\MetricParameter]$ is $\PileHight[\NumIterations]$. Here it
    does not matter whether $\LipMaps[\MetricParameter]$ is endowed with
    the norm $\NormOf[\infty]{\cdot}$ or
    $\InnerNormOf[\MetricParameter]{\cdot}$.
  \end{observation}

  \begin{corollary}\label{contraction}
    The spectral radius of $\Tf$ on $\LipMaps[\MetricParameter]$ as
    given by Gelfand's formula, is
    \(
      \TheRadius := \Lim[\NumIterations\rightarrow\infty]{
        \NormOf[\op]{\Tf^{\NumIterations}}^{1/\NumIterations}
      }
      =
      \Lim[\NumIterations\rightarrow\infty]{
        \PileHight[\NumIterations]^{1/\NumIterations}
      }
      \leq
      \MaxModDegree
    \).

    For any $\TheEps>0$, we have
    \begin{equation}\label{eq:key_inequality}
      \InnerNormOf[\MetricParameter]{\Tf^{\NumIterations} \FnPath}
      \leq
      (\TheRadius+\TheEps)^{\NumIterations}
      (
      \MetricParameter^{\NumIterations}
      \InnerConstOf[\MetricParameter]{\FnPath}
      +
      \NormOf[\infty]{\FnPath}
      )      
    \end{equation}
    for all $\NumIterations$ sufficiently large.
  \end{corollary}    
  Note that \eqref{eq:key_inequality} is called \emph{Key inequality}
  in \cite[Lemma~1.2]{Baladi}. We will see that this inequality is indeed the key to deduce the quasi compactness of $\Tf$ on
  $\LipMaps[\MetricParameter]$.

  \section{Approximation by projection}
  We keep the assumption
  that there is a uniform upper bound for the degree of vertices in
  $\TheGraph$ and the notation
  \(
    \MaxModDegree :=
    \Max{
      \SetOf[\ModDegreeOf{\TheVertex}]{
        \TheVertex\in\VertexSet
      }
    }
  \).

  For each postal code $\TheCode$ of length $\TheLevel$, pick a point
  $\Prefer{\TheCode}$ in the district $\PathSet[\TheCode]$, and define
  \[
    \TheProjection[\TheLevel]\mapcolon\LipMaps[\MetricParameter]
    \longrightarrow \Depends[\TheLevel]
  \]
  by the requirement that $\TheProjection[\TheLevel]\FnPath$
  is constant on all level-$\TheLevel$ districts and 
  coincides with $\FnPath$ on the points $\Prefer{\TheCode}$ for
  all postal codes of length $\TheLevel$.

  \begin{observation}
    As the diameter of a level-$\TheLevel$ district is
    $\MetricParameter^{\TheLevel}$, we have
    \[
      \NormOf[\infty]{\FnPath - \TheProjection[\TheLevel]\FnPath}
      \leq
      \MetricParameter^{\TheLevel}
      \InnerConstOf[\MetricParameter]{\FnPath}
    \]
    by Lipschitz continuity. On the other hand,
    \[
      \InnerConstOf[\MetricParameter]{\TheProjection[\TheLevel]\FnPath}
      \leq
      \InnerConstOf[\MetricParameter]{\FnPath}     
    \]
    as the distance of points in different level-$\TheLevel$ districts does
    only depend on their $\TheLevel$-digit postal codes.
  \end{observation}
  \begin{corollary}\label{iteratedTf}
    For any $\TheEps>0$ the following inequalities hold for all
    $\NumIterations$ sufficiently large:
    \begin{align*}
      \NormOf[\infty]{
      \Tf^{\NumIterations}\FnPath
      -
      \Tf^{\NumIterations}\TheProjection[\NumIterations]\FnPath
      }
      & 
        \leq
        (\TheRadius+\TheEps)^{\NumIterations}
        \MetricParameter^{\NumIterations}
        \InnerConstOf[\MetricParameter]{\FnPath}      
      \\
      \InnerNormOf[\MetricParameter]{
      \Tf^{\NumIterations}\FnPath
      -
      \Tf^{\NumIterations}\TheProjection[\NumIterations]\FnPath
      }
      &
        \leq
        (\TheRadius+\TheEps)^{\NumIterations}
        \MetricParameter^{\NumIterations}
        \InnerConstOf[\MetricParameter]{\FnPath}
    \end{align*}
  \end{corollary}
  \begin{proof}
    Using the previous observation, the first claim is immediate from:
    \[
      \NormOf[\infty]{\Tf^\NumIterations(\FnPath- \TheProjection[\NumIterations] \FnPath)}
      \leq R_m\NormOf[\infty]{\FnPath-\TheProjection[\NumIterations]\FnPath}
      \leq (\TheRadius+\TheEps)^\NumIterations \MetricParameter^\NumIterations
      \InnerConstOf[\MetricParameter]{\FnPath}
    \]
    The second follows from the observation in conjunction with
    Corollary~\ref{contraction}, applied to $\frac{\TheEps}{2}$. We obtain
    using, the Key inequality \eqref{eq:key_inequality},
    \begin{align*}
      \InnerNormOf[\MetricParameter]{
      \Tf^{\NumIterations}\FnPath
      -
      \Tf^{\NumIterations}\TheProjection[\NumIterations]\FnPath
      }
      &
        \leq
        (\TheRadius+\frac{\TheEps}{2})^{\NumIterations}
        (
        \MetricParameter^{\NumIterations}
        \InnerConstOf[\MetricParameter]{
        \FnPath - \TheProjection[\TheLevel]{\FnPath}
        }
        +
        \NormOf[\infty]{
        \FnPath - \TheProjection[\TheLevel]{\FnPath}
        }        
        )
      \\
      &
        \leq
        3
        (\TheRadius+\frac{\TheEps}{2})^{\NumIterations}
        \MetricParameter^{\NumIterations}
        \InnerConstOf[\MetricParameter]{\FnPath}
      \\
      &
        \leq
        (\TheRadius+\TheEps)^{\NumIterations}
        \MetricParameter^{\NumIterations}
        \InnerConstOf[\MetricParameter]{\FnPath}
    \end{align*}
    for sufficiently large $\NumIterations$.
  \end{proof}

  \begin{proposition}\label{approximation-by-subspaces}
    Let $\TheTwist\in\ComplexNumbers$ satisfy $\AbsValueOf{\TheTwist} >
    \MetricParameter \TheRadius$. Then, the equalizer
    \(
      \Equalizer{\Tf}{\TheTwist\Identity}
    \)
    on $\LipMaps[\MetricParameter]$ is a subspace of $\Depends[1]$.
  \end{proposition}
  \begin{proof}
    Choose $\TheEps>0$ such that
    $\AbsValueOf{\TheTwist}>(\MetricParameter+\TheEps)(\TheRadius+\TheEps)$.
    Note that $\NormOf[\op]{\Tf^\NumIterations}
    \leq(\TheRadius+\TheEps)^\NumIterations$ for all sufficiently large
    $\NumIterations$.

    Consider $\FnPath$ in the equalizer
    \(
      \Equalizer{\Tf}{\TheTwist\Identity}
    \).
    Put
    \(
      \FnPath[\TheLength] := \TheProjection[\TheLength]\FnPath
    \)
    and
    \(
      \FnDiff[\TheLength]:=\FnPath-\FnPath[\TheLength]
      .
    \)
    
    As $\Tf\FnPath=\TheTwist\FnPath$, we find:
    \begin{align*}
      \FnPath
      & =
        \TheTwist^{-\TheLength}\Tf^{\TheLength}\FnPath
      \\
      & =
        \TheTwist^{-\TheLength}\Tf^{\TheLength}\FnPath[\TheLength]
        +
        \TheTwist^{-\TheLength}\Tf^{\TheLength}\FnDiff[\TheLength]        
    \end{align*}
    Note that
    \(
      \TheTwist^{-\TheLength}\Tf^{\TheLength}\FnPath[\TheLength]\in\Depends[1]
      .
    \)
    From the previous corollary, we have
    \[
      \NormOf[\infty]{\TheTwist^{-\TheLength}\Tf^{\TheLength}\FnDiff[\TheLength]}
      \leq
      \frac{
        (\TheRadius+\TheEps)^{\TheLength}
        \MetricParameter^{\TheLength}
        \InnerConstOf[\MetricParameter]{\FnPath}
      }{
        (\TheRadius+\TheEps)^{\TheLength}
        (\MetricParameter+\TheEps)^{\TheLength}
      }
      \leq
      \frac{\MetricParameter^{\TheLength}}{(\MetricParameter+\TheEps)^{\TheLength}}
      \InnerConstOf[\MetricParameter]{\FnPath}
    \]
    for sufficiently large $\TheLength$.
    It follows that $\FnPath$ is the $\NormOf[\infty]{\cdot}$-limit of the
    sequence
    \(
      (\TheTwist^{-\TheLength}\Tf^{\TheLength}\FnPath[\TheLength])
    \)
    which stays in the closed subspace $\Depends[1]$. Hence
    $\FnPath\in\Depends[1]$.
  \end{proof}

  \section{The case of finite graphs}
  For a finite graph $\TheGraph$, the spaces $\Depends[\TheLength]$
  have finite dimension and we can use the estimates from
  Corollary~\ref{iteratedTf} to bound the essential spectral
  radius.
  \begin{proposition}\label{essential-spectrum}
    Assume that the graph $\TheGraph$ is finite. Then the essential
    spectral radius $\TheRadius[\ess]$ of
    \(
      \Tf[\MetricParameter]:=\Tf|_{\LipMaps[\MetricParameter]}
    \)
    does not exceed $\MetricParameter\TheRadius$.
  \end{proposition}
  \begin{proof}
    If $\TheGraph$ is finite, the spaces $\Depends[\TheLength]$ all
    have finite dimension. Consequently, the projection operators
    \(
      \FiniteRank[\NumIterations]
      :=
      \Tf[\MetricParameter]^{\NumIterations}
      \TheProjection[\NumIterations]
    \)
    have finite rank. 
    
    Corollary~\ref{iteratedTf} and
    \(
      \InnerConstOf[\MetricParameter]{\FnPath}
      \leq
      \InnerNormOf[\MetricParameter]{\FnPath}
    \)
    imply that for any $\varepsilon>0$ the operator norm of
    \(
      \Tf[\MetricParameter]^{\NumIterations}-\FiniteRank[\NumIterations]
    \)
    is bounded from above by
    $\MetricParameter^{\NumIterations}(\TheRadius+\TheEps)^{\NumIterations}$
    for sufficiently large $\NumIterations$. Consequently for such
    $\NumIterations$ the essential spectral radius of
    $\Tf[\MetricParameter]^\NumIterations$ is bounded by
    $\MetricParameter^{\NumIterations}(\TheRadius+\TheEps)^{\NumIterations}$
    and by standard spectral theory one deduces that the
    essential spectral radius of $\Tf[\MetricParameter]$ is bounded by
    $\MetricParameter(\TheRadius+\TheEps)$.

    Here are the details. First, note that if the essential spectral
    radius of $\Tf[\MetricParameter]^\NumIterations$ on a Banach space
    $\TheBanachSpace$ is $\TheRadius[0]$, then $\TheBanachSpace$
    decomposes as a direct sum
    \[
      \TheBanachSpace=
      \TheCore
      \directsum (\DirectSum[
      \TheEigenvalue\in\SpectrumOf{\Tf[\MetricParameter]^{\NumIterations}},
      \AbsValueOf{\TheEigenvalue} > {\TheRadius[0]}]{
        \TheEigenspace[\TheEigenvalue]
      })
    \]
    where $\TheEigenspace[\TheEigenvalue]$ is the finite dimensional
    generalized eigenspace $\ker(\Tf[\MetricParameter]^\NumIterations-\TheEigenvalue)^J$
    and there are only finitely
    many $\TheEigenvalue$ occurring in the sum.  As $\Tf[\MetricParameter]$ commutes with
    $\Tf[\MetricParameter]^\NumIterations$ all these $\TheEigenspace[\TheEigenvalue]$ are
    $\Tf[\MetricParameter]$ invariant and as they are finite dimensional,
    $(\Tf[\MetricParameter]-\ComplexVar)^{-1}$ is meromorphic with finite rank residues on
    all $\TheEigenspace[\TheEigenvalue]$.

    The remaining space $\TheCore$ is
    $\Tf[\MetricParameter]^\NumIterations$ invariant and the spectral
    radius of $\Tf[\MetricParameter]^\NumIterations$ is $\leq
    \TheRadius[0]$ when restricted to $\TheCore$. But this implies
    directly that for $\TheCoreVector\in \TheCore$ the Neumann series of
    the resolvent
    \(
      (\Tf[\MetricParameter]-\ComplexVar)^{-1}\TheCoreVector
      =
      \sum_{k=0}^\infty -\frac{\Tf[\MetricParameter]^k}{\ComplexVar^k+1}\TheCoreVector
    \)
    converges for $\AbsValueOf{\ComplexVar}\geq
    \TheRadius[0]^{1/\NumIterations}$ and is therefore holomorphic in this
    domain. Thus we have shown that the resolvent
    $(\Tf[\MetricParameter]-\ComplexVar)^{-1}:\TheBanachSpace\to
    \TheBanachSpace$ is meromorphic with finite rank residues for
    $\AbsValueOf{\ComplexVar}>\TheRadius[0]^{1/\NumIterations}$. So the
    spectrum of $\Tf[\MetricParameter]$ is discrete for
    $\AbsValueOf{\ComplexVar}>\TheRadius[0]^{1/\NumIterations}$ and by
    holomorphic functional calculus we know that the spectral projectors
    on the spectral values at these points are given by the residues and
    are thus finite dimensional. In this way we have bounded the essential
    spectral radius of $\Tf[\MetricParameter]$ by
    $\TheRadius[0]^{1/\NumIterations}$.
  \end{proof}
  \begin{remark}\label{rem:Ruelle resonance}
    Proposition~\ref{essential-spectrum} is a well known fact in the
    setting of subshifts of finite type (see e.g. \cite[Chapter~1]{Baladi}
    for a comprehensive treatment). The discrete spectrum appearing for
    $\AbsValueOf{\TheTwist}\geq\MetricParameter\TheRadius$ is called the
    spectrum of \emph{Ruelle resonances} of the transfer operator $\mathcal L_\vartheta$.
  \end{remark}

  Thus, in this case, we have a very satisfying picture; and our
  discussion so far can be summarized as follows:
  \begin{theorem}\label{main}
    Fix a parameter $\MetricParameter\in(0,1)$.  Let $\TheGraph$ be a
    connected finite graph without terminal vertices.  The transfer
    operator $\Tf[\MetricParameter]$ on $\LipMaps[\MetricParameter]$
    is bounded.  Its spectral radius
    \(
      \TheRadius =
      \Lim[\NumIterations\rightarrow\infty]{        
        \NormOf[\op]{\Tf[\MetricParameter]^{\NumIterations}}^{1/\NumIterations}
      }
    \)
    does not depend on $\MetricParameter$ and is
    bounded from above by the maximum
    \(
      \MaxModDegree :=
      \Max{
        \SetOf[\ModDegreeOf{\TheVertex}]{
          \TheVertex\in\VertexSet
        }
      }
    \)
    where equality holds if  and only if $\TheGraph$ is a regular graph
    (i.e., all vertices have the same degree).

    Furthermore, any $\TheEigenvalue$ in the spectrum of the transfer
    operator $\Tf[\MetricParameter]$ satisfying
    $\AbsValueOf{\TheEigenvalue}>\MetricParameter\TheRadius$ is an
    eigenvalue of finite multiplicity and the corresponding eigenspace
    $\TheEigenspace[\TheEigenvalue]$ is contained in the subspace
    $\Depends[1]$. It is isomorphic to the eigenspace
    $\Equalizer{\TurnSum}{\TheEigenvalue\Identity}$ and for
    $\TheEigenvalue\neq\pm 1$ also isomorphic to the equalizer
    \(
      \Equalizer{\NeighborSum}{\Rescale[\TheEigenvalue]}
    \),
    which can also be described as
    \(
      \Equalizer{\NeighborAvg}{\LocalTweak[\TheTwist]}
    \).
  \end{theorem}
  \begin{proof}
    By Lemma~\ref{rescaling-bounds}, the operator $\Tf[\MetricParameter]$ is
    bounded.  Its spectral radius $\TheRadius$ has been determined in
    Corollary~\ref{contraction} in terms of the numbers
    $\PileHightOf[\NumIterations]{\TheEdge}$ of edge paths of length
    $\NumIterations$ that end with the edge $\TheEdge$. Recall that
    \[
      \TheRadius
      =
      \Lim[\NumIterations\rightarrow\infty]{
        \PileHight[\NumIterations]^{1/\NumIterations}
      }
    \]
    where
    \(
      \PileHight[\NumIterations]
      =
      \Max{
        \SetOf[{\PileHightOf[\NumIterations]{\TheEdge}}]{
          \TheEdge\in\EdgeSet
        }
      }
    \).
    For a regular graph, $\MaxModDegree=\ModDegreeOf{\TheVertex}$ for
    any vertex $\TheVertex$, whence we find
    $\PileHightOf[\NumIterations]{\TheEdge}=\MaxModDegree^\NumIterations$
    for any edge. Hence, $\TheRadius=\MaxModDegree$.
    
    For a non-regular graph and $\MoreIterations$ large enough, one
    has $\PileHight[\MoreIterations]<\MaxModDegree^{\MoreIterations}$. To
    see that the spectral radius is strictly less than $\MaxModDegree$, it
    just remains to observe
    \(
      \PileHightOf[\MoreIterations+\NumIterations]{\TheEdge}
      \leq
      \PileHight[\MoreIterations]
      \PileHightOf[\NumIterations]{\TheEdge}
    \)
    whence
    \(
      \PileHight[\MoreIterations+\NumIterations] \leq
      \PileHight[\MoreIterations] \PileHight[\NumIterations]
    \).

    If $\TheEigenvalue$ lies in the spectrum of
    $\Tf[\MetricParameter]$ and satisfies
    $\AbsValueOf{\TheEigenvalue}>\MetricParameter\TheRadius$, then it lies
    in the essential spectrum by Proposition~\ref{essential-spectrum},
    i.e., it is an eigenvalue of finite multiplicity.
    
    Also for $\AbsValueOf{\TheEigenvalue}>\MetricParameter\TheRadius$,
    the eigenspace $\TheEigenspace[\TheEigenvalue]$ is a subspace of
    $\Depends[1]=\MapsOn{\EdgeSet}$ by
    Proposition~\ref{approximation-by-subspaces}. Note that
    $\TheEigenvalue\neq 0$ because of
    $\AbsValueOf{\TheEigenvalue}>\MetricParameter\TheRadius$. The transfer
    operator restricted to $\Depends[1]$ agrees with the turn sum operator
    $\TurnSum$ on $\MapsOn{\EdgeSet}$. Hence
    \(
      \TheEigenspace[\TheEigenvalue]
      =
      \Equalizer{\TurnSum}{\TheEigenvalue\Identity}
    \).
    If, in addition, $\TheEigenvalue\neq\pm 1$,
    Proposition~\ref{algebraic} applies and
    $\Equalizer{\TurnSum}{\TheEigenvalue\Identity}$ is isomorphic to the
    equalizer $\Equalizer{\NeighborSum}{\Rescale[\TheEigenvalue]}$.
  \end{proof}

  This result has a consequence for the space
  \(
    \Depends[\infty]
    :=
    \Depends[1]\union\Depends[2]\union\cdots
  \)
  of functions on $\PathSet$ that depend only on an initial segment.
  \begin{corollary}\label{transfer-on-loc-const}
    For $\TheTwist\neq 0$, the eigenspace
    $\Equalizer{\Tf|_{\Depends[\infty]}}{\TheTwist\Identity}$ of $\Tf$
    restricted to $\Depends[\infty]$ is contained in $\Depends[1]$ and is
    isomorphic to the eigenspace
    $\Equalizer{\TurnSum}{\TheTwist\Identity}$. For
    $\TheTwist\not\in\SetOf{-1,0,1}$, Proposition~\ref{algebraic}, yields
    an additional isomorphism to
    \(
      \Equalizer{\NeighborSum}{\Rescale[\TheTwist]}
    \),
    i.e., to
    \(
      \Equalizer{\NeighborAvg}{\LocalTweak[\TheTwist]}
    \).
  \end{corollary}
  \begin{proof}
    As above, we put
    \(
      \MaxModDegree :=
      \Max{
        \SetOf[\ModDegreeOf{\TheVertex}]{
          \TheVertex\in\VertexSet
        }
      }
    \).
    Since we assume $\TheTwist\neq 0$, we can choose
    $\MetricParameter\in(0,1)$ such that
    $\MetricParameter\MaxModDegree<\AbsValueOf{\TheTwist}$. Now,
    Theorem~\ref{main} applies as
    $\Tf|_{\Depends[\infty]}=\Tf[\MetricParameter]|_{\Depends[\infty]}$.
  \end{proof}

  \begin{remark}\label{rem:correspondence}
    These results should be compared to the quantum classical
    correspondence for geodesic flows: The scale of Lipschitz continuous
    functions on the edge paths are an analogue to the anisotropic Sobolev
    spaces on the sphere bundle of a negatively curved manifold. The
    quantum classical correspondence on manifolds of constant negative
    curvature \cite{DFG15, GHW18a} (or locally symmetric spaces
    \cite{KW21, GHW21, HHW21}) says that there is bijection between a
    subset of Ruelle resonances (so called first band resonances) of the
    geodesic flow and the eigenvalues of the Laplacian on the locally
    symmetric space. In our case of graphs we obtain the same
    relation to the eigenvalues of the Laplacian if the graph is regular,
    which can be interpreted as having constant curvature or being a
    locally symmetric space.
    In fact, while in Remark~\ref{rem:Ruelle resonance} the set of
    Ruelle resonances depends $\vartheta$, choosing $\vartheta$ small
    enough we do indeed get the entire Laplace spectrum via the
    correspondence.
    If we relax the assumption of regularity of
    the graph to the assumption of bounded degree, then we obtain not an
    eigenfunction of the Laplacian anymore but rather an element in the
    equalizer $\Equalizer{\NeighborAvg}{\LocalTweak[\TheTwist]}$. But as
    $\LocalTweak[\TheTwist]$ is a multiplication operator these are
    precisely eigenfunctions of the Laplacian plus a lower order remainder
    term.
    This in turn can directly be compared to the recent work
    of Faure and Tsujii \cite{FT21} where they establish some form of
    quantum classical correspondence for geodesic flows of variable
    curvature. In their work, the Laplace operator is also perturbed by
    (rather inexplicit) lower order terms.
  \end{remark}

  \section{Interpretation of the dual}
  We continue the discussion of a finite graph $\TheGraph$.  In this
  case, the space $\PathSet$ is compact and we can recognize
  $\Depends[\infty]$ as the space
  \(
    \LocConstOn{\PathSet}
  \)
  of locally constant functions on $\PathSet$. A function
  $\FnPath\mapcolon\PathSet\rightarrow\ComplexNumbers$ is locally
  constant if there is a cover of $\PathSet$ by districts on each of
  which $\FnPath$ is constant.  As $\PathSet$ is compact, a finite cover
  by such districts exists and the function is contained in
  $\Depends[\TheLevel]$ for any level exceeding the deepest level used
  in the cover.

  As $\Depends[\infty]=\LocConstOn{\PathSet}$ is contained in
  $\LipMaps[\MetricParameter]$ for arbitrary small $\MetricParameter$,
  we denote the restriction of the transfer operator $\Tf$ to
  $\LocConstOn{\PathSet}$ by $\AlgTf$.  We shall give a combinatorial
  description for the algebraic dual of $\LocConstOn{\PathSet}$ and
  describe the dual transfer operator $\DualAlgTf$ defined on the
  algebraic dual.

  Let $\PrefixSet[\TheLevel]$ be the set of $\TheLevel$-digit postal
  codes, or geometrically the set of edge paths of length $\TheLevel$
  without backtracking. As $\TheGraph$ is finite, these are finite
  sets. We consider the disjoint union, i.e. the set
  \[
    \PrefixSet := \PrefixSet[1] \union \PrefixSet[2] \union \cdots
  \]
  of all postal codes (or finite length edge paths without
  backtracking). A postal code $\TheCode$ designates a district. Each
  one-digit-extension $\TheCode\turn\TheEdge$ designates an immediate
  sub-district. We define the space
  \[
    \FinAddMeasuresOn{\PathSet}
    :=
    \SetOf[
    \TheMeasure \mapcolon \PrefixSet \rightarrow \ComplexNumbers
    ]{
      \TheMeasureOf{\TheCode}
      =
      \Sum[\TheCode\turn\TheEdge]{
        \TheMeasureOf{\TheCode\turn\TheEdge}
        \text{\ for all\ }\TheCode\in\PrefixSet
      }
    }
  \]
  of all complex-valued function on $\PrefixSet$ that are additive
  with respect to summing over all one-digit-extensions. Such a function
  defines a complex valued finitely additive measure on the clopen sets of the totally
  disconnected compact space $\PathSet$. 

  Given $\TheMeasure \mapcolon\PrefixSet\rightarrow\ComplexNumbers$ in
  $\FinAddMeasuresOn{\PathSet}$ and a function
  $\FnPath\in\Depends[\TheLength]$ on $\PathSet$ that is constant
  on level-$\TheLength$ districts, we define the pairing
  \begin{equation}\label{pairing}
    \PairingOf{\FnPath}{\TheMeasure}
    :=
    \Int[\ThePath\in\PathSet]{
      \FnPathOf{\ThePath}
      \diff\TheMeasureOf{\ThePath}
    }
    :=
    \Sum[{\TheCode\in{\PrefixSet[\TheLength]}}]{
      \FnPathOf{\TheCode}
      \TheMeasureOf{\TheCode}
    }
  \end{equation}
  where $\FnPathOf{\TheCode}$ is the value of $\FnPath$ across the
  district designated by $\TheCode$. As $\TheGraph$ is finite, this is a
  finite sum. The pairing is well defined: regarding $\FnPath$ as an
  element of $\Depends[\TheLength+1]$ does not affect the value
  $\PairingOf{\FnPath}{\TheMeasure}$ by the additivity of $\TheMeasure$.

  \begin{proposition}
    Assume that $\TheGraph$ is finite.  The map
    \(
      \TheMeasure\mapsto\PairingOf{\DummyArg}{\TheMeasure}
    \)
    identifies $\FinAddMeasuresOn{\PathSet}$ with the algebraic dual of
    $\LocConstOn{\PathSet}=\Depends[\infty]$.
  \end{proposition}
  
  \begin{proof}
    We have discussed the case where $\PathSet$ is the space of ends
    of a tree in \cite[Section~3]{BHW22}.  The present case, where
    $\PathSet$ is the space of ends of a forest with finitely many roots,
    does not differ substantially.
  \end{proof}

  Let $\CharFct[\TheCode]$ be the characteristic function of the
  district designated by $\TheCode$, i.e.,
  \(
    \CharFctOf[\TheCode]{\ThePath} = 1
  \)
  if $\TheCode$ is an initial prefix of $\ThePath$, and
  \(
    \CharFctOf[\TheCode]{\ThePath} = 0
  \)
  otherwise. We can compute the effect of the transfer operator:
  \begin{equation}\label{tf-on-char-fct}
    \AlgTfOf{\CharFct[\TheEdge\turn\TheCode]} = \CharFct[\TheCode]
  \end{equation}
  To see this, we evaluate on a path:
  \[
    (\AlgTf\CharFct[\TheEdge\turn\TheCode])(\ThePath) =
    \Sum[\AltEdge\turn\ThePath]{
      \CharFctOf[\TheEdge\turn\TheCode]{\AltEdge\turn\ThePath}
    }
    =
    \CharFctOf[\TheCode]{\ThePath}
  \]
  As a consequence, we can also compute the effect of the dual transfer
  operator $\DualAlgTf$ on finitely additive measures:
  \begin{equation}\label{dual-tf-on-measure}
    (\DualAlgTf\TheMeasure)(\TheEdge\turn\TheCode)
    = \PairingOf{
      \CharFct[\TheEdge\turn\TheCode]
    }{
      \DualAlgTfOf{\TheMeasure}
    }
    =
    \PairingOf{
      \AlgTfOf{
        \CharFct[\TheEdge\turn\TheCode]
      }
    }{
      \TheMeasure
    }
    = \PairingOf{\CharFct[\TheCode]}{\TheMeasure}
    =
    \TheMeasureOf{\TheCode}
  \end{equation}

  Furthermore, formula~\pref{dual-tf-on-measure} interacts very nicely with the
  eigenspace identity $\DualAlgTfOf{\TheMeasure}=\TheTwist\TheMeasure$.
  
  \begin{observation}\label{obs:measure-eigenspace}
    A measure $\TheMeasure\in\FinAddMeasuresOn{\PathSet}$ belongs to the
    eigenspace $\Equalizer{\DualAlgTf}{\TheTwist\Identity}$ if and only
    if
    \(
      \TheMeasureOf{\TheCode}
      =
      (\DualAlgTf\TheMeasure)(\TheEdge\turn\TheCode)
      =
      \TheTwist\TheMeasureOf{\TheEdge\turn\TheCode}
    \)
    for all postal codes $\TheEdge\turn\TheCode$.

    In particular, only the everywhere vanishing measure satisfies
    this condition for $\TheTwist=0$, i.e., the dual transfer operator
    $\DualAlgTf$ is injective.
  \end{observation}
  For $\TheTwist\neq 0$, we deduce immediately that for
  $\TheMeasure\in\Equalizer{\DualAlgTf}{\TheTwist\Identity}$, the value of
  $\TheMeasure$ on any path only depends on its length and its final
  edge:
  \begin{equation}\label{action-at-a-distance}
    \TheMeasureOf{\TheEdge[1]\turn\cdots\turn\TheEdge[\TheLength]\turn\AltEdge}
    =
    \TheTwist^{-n}\TheMeasureOf{\AltEdge}
  \end{equation}
  In particular, such a $\TheMeasure$ is completely determined by its values
  on $\PrefixSet[1]$.

  \begin{theorem}\label{dual-main}
    For a finite graph $\TheGraph$ and $\TheTwist\neq 0$, the
    eigenspace $\Equalizer{\DualAlgTf}{\TheTwist\Identity}$ in
    $\FinAddMeasuresOn{\PathSet}$ is isomorphic to the eigenspace
    $\Equalizer{\TurnSum}{\TheTwist\Identity}$ in
    $\MapsOn{\EdgeSet}$.

    More precisely, if a measure $\TheMeasure$
    lies in
    $\Equalizer{\DualAlgTf}{\TheTwist\Identity}$, then the function
    \(
      \FnEdge\mapcolon\TheEdge \mapsto \TheMeasureOf{\OppositeOf{\TheEdge}}
    \)
    lies in $\Equalizer{\TurnSum}{\TheTwist\Identity}$. Conversely, if
    $\FnEdge\in\Equalizer{\TurnSum}{\TheTwist\Identity}$, the assignment
    \[
      \TheMeasureOf{\TheEdge[1]\turn\cdots\turn\TheEdge[\TheLength]}
      :=
      \TheTwist^{1-\TheLength}\FnEdgeOf{\OppositeOf{\TheEdge[\TheLength]}}
    \]
    defines a measure in $\Equalizer{\DualAlgTf}{\TheTwist\Identity}$.
  \end{theorem}
  
  \begin{proof}
    First assume $\TheMeasure\in\Equalizer{\DualAlgTf}{\TheTwist\Identity}$.
    By finite additivity, we have
    \[
      \TheTwist\FnEdgeOf{\TheEdge}
      =
      \TheTwist\TheMeasureOf{\OppositeOf{\TheEdge}} =
      \Sum[\OppositeOf{\TheEdge}\turn\OppositeOf{\AltEdge}]{
        \TheTwist\TheMeasureOf{\OppositeOf{\TheEdge}\turn\OppositeOf{\AltEdge}}
      }
      =
      \Sum[\OppositeOf{\TheEdge}\turn\OppositeOf{\AltEdge}]{
        \TheMeasureOf{\OppositeOf{\AltEdge}}
      }
      =
      \Sum[\AltEdge\turn\TheEdge]{
        \FnEdgeOf{\AltEdge}
      }
      =
      (\TurnSum\FnEdge)(\TheEdge)
    \]
    which is $\FnEdge\in\Equalizer{\TurnSum}{\TheTwist\Identity}$.

    Now assume
    $\FnEdge\in\Equalizer{\TurnSum}{\TheTwist\Identity}$. Reading the
    above equation in reverse, we find that $\TheMeasure$
    satisfies additivity for single digit postal codes:
    \[
      \TheMeasureOf{\TheEdge}
      =
      \Sum[\TheEdge\turn\AltEdge]{
        \TheMeasureOf{\TheEdge\turn\AltEdge}
      }
    \]
    Additivity for longer postal codes follows immediately from the
    given rule:
    \[
      \TheMeasureOf{\TheEdge[1]\turn\cdots\turn\TheEdge[\TheLength]}
      =
      \TheTwist^{1-n}\TheMeasureOf{\TheEdge[\TheLength]}
      =
      \TheTwist^{1-n}
      \Sum[{
        \TheEdge[\TheLength]\turn\AltEdge
      }]{
        \TheMeasureOf{\TheEdge[\TheLength]\turn\AltEdge}
      }
      =
      \Sum[{
        \TheEdge[\TheLength]\turn\AltEdge
      }]{
        \TheMeasureOf{\TheEdge[1]\turn\cdots\turn\TheEdge[\TheLength]\turn\AltEdge}
      }
    \]
    This completes the proof.
  \end{proof}

  \begin{corollary}\label{canonical-transpose}
    A finitely additive measure $\TheMeasure$ lies in the eigenspace
    $\Equalizer{\DualAlgTf}{\TheTwist\Identity}$ if and only if the map
    \(
      (\TheEdge[1]\turn\TheEdge[2]\turn\cdots)
      \mapsto
      \TheMeasureOf{\OppositeOf{\TheEdge[1]}}
    \),
    lies in the eigenspace
    $\Equalizer{\AlgTf}{\TheTwist\Identity}$. This yields an explicit
    isomorphism between the $\TheTwist$-eigenspaces of $\DualAlgTf$ and
    $\AlgTf$.
  \end{corollary}
  We shall show that this isomorphism
  extends to the $\TheTwist$-eigenspaces of
  $\DualTf[\MetricParameter]$ and $\Tf[\MetricParameter]$ provided that
  $\MetricParameter$ is small enough.

  \begin{remark}
    Theorem~\ref{dual-main} is where we use a very special property of
    our setting: the language $\PrefixSet$ of postal codes has the
    interesting feature that flipping each letter in a word to its
    opposite yields the reverse of a word that is also in the language.
  \end{remark}
  \begin{remark}
    The turn sum operator $\TurnSum$ on a locally finite graph is
    injective.  Consider a function
    $\FnEdge\in\MapsOn{\EdgeSet}$ in the kernel of
    $\TurnSum$.  That $\TurnSumOf{\FnEdge}$ vanishes on all outgoing edges
    around a vertex yields a fully determined linear system for the values
    of $\FnEdge$ on the incoming edges, which has only the trivial
    solution. Thus, in view of Observation~\ref{obs:measure-eigenspace} the
    eigenspaces from Theorem~\ref{dual-main} are also isomorphic in the
    case $\TheTwist=0$.
  \end{remark}
  
  \begin{remark}
    The combinatorial descriptions of $\FinAddMeasuresOn{\PathSet}$
    and the eigenspaces $\Equalizer{\TurnSum}{\TheTwist\Identity}$ suggest
    that it might be possible to find a direct argument for
    Corollary~\ref{transfer-on-loc-const} without introducing the
    Lipschitz-type metrics depending on the parameter $\MetricParameter$.
  \end{remark}

  \begin{observation}\label{good-choice}
    We made a choice of preferred continuations
    $\Prefer{\TheCode}\in\PathSet[\TheCode]$ for all finite edge paths
    $\TheCode$. A particular way of doing that is to pick a preferred
    outgoing turn for each edge and continue paths along such preferred
    turns. This choice has the consequence, that
    \[
      \Prefer{(\TheEdge\turn\TheCode)} = \TheEdge\turn\Prefer{\TheCode}
    \]
    for any path $\TheEdge\turn\TheCode$. As a consequence for the
    projection operators $\TheProjection[\TheLevel]$ we find a
    particularly nice compatibility with the transfer operator, namely:
    \[
      \TheProjection[\TheLevel]\AlgTf = \AlgTf\TheProjection[\TheLevel+1].
    \]
  \end{observation}
  \begin{theorem}
    For
    \(
      \MetricParameter\in
      (0,\frac{\AbsValueOf{\TheTwist}}{\MaxModDegree})
    \),
    the eigenspace $\Equalizer{\DualAlgTf}{\TheTwist\Identity}$ is contained
    in the continuous dual $\ContDual{\LipMaps[\MetricParameter]}$ and
    equals the eigenspace
    $\Equalizer{\DualTf[\MetricParameter]}{\TheTwist\Identity}$.
  \end{theorem}
  \begin{proof} (compare \cite[Lemma~5.11]{BHW22} and its proof)
    Let $\TheMeasure\in\Equalizer{\DualAlgTf}{\TheTwist\Identity}$. We
    have to extend the pairing $\PairingOf{\DummyArg}{\TheMeasure}$ from
    $\Depends[\infty]$ to the function space $\LipMaps[\MetricParameter]$
    and show that the extension is continuous.

    To define an extension, we use the projection operators
    \(
      \TheProjection[\TheLevel]\mapcolon\LipMaps[\MetricParameter]
      \rightarrow \Depends[\TheLevel]
    \)
    from before.

    For $\FnPath\mapcolon\PathSet\rightarrow\ComplexNumbers$, put
    \(
      \FnPath[\TheLevel]:=
      \TheProjectionOf[\TheLevel]{\FnPath}
      \in\Depends[\TheLevel]
    \).
    We claim that the sequence
    $(\PairingOf{\FnPath[\TheLevel]}{\TheMeasure})_m$ converges. To see this,
    let us compute the distance of consecutive terms. First we have by
    additivity of the measure:
    \[
      \PairingOf{
        \FnPath[\TheLevel] - \FnPath[\TheLevel+1]
      }{\TheMeasure}
      =
      \Sum[{\TheCode\in\PrefixSet[\TheLevel]}]{
        \Sum[{\TheCode\turn\TheEdge\in\PrefixSet[\TheLevel+1]}]{
          (\FnPathOf{\Prefer{\TheCode}}
          -
          \FnPathOf{\Prefer{(\TheCode\turn\TheEdge)}})
          \TheMeasureOf{\TheCode\turn\TheEdge}
        }
      }
    \]
    Note that there are at most
    $(\CardOf{\VertexSet})\MaxModDegree^{\TheLevel+1}$ summands. Thus, we
    find
    \[
      \AbsValueOf{
        \PairingOf{
          \FnPath[\TheLevel] - \FnPath[\TheLevel+1]
        }{\TheMeasure}
      }
      \leq
      \InnerConstOf[\MetricParameter]{\FnPath}
      (\CardOf{\VertexSet})
      \MaxModDegree^{\TheLevel+1}
      \MetricParameter^{\TheLevel}
      \AbsValueOf{\TheTwist}^{-\TheLevel}
      \Max[\TheEdge\in\EdgeSet]{
        \AbsValueOf{ \TheMeasureOf{\TheEdge} }
      }
    \]
    which tends to $0$ exponentially fast provided that
    $\frac{\MetricParameter\MaxModDegree}{\AbsValueOf{\TheTwist}}<1$.
    Consequently the sequence
    $(\PairingOf{\FnPath[\TheLevel]}{\TheMeasure})_{\TheLevel}$ converges.

    Thus, we put
    \[
      \PairingOf{ \FnPath }{ \TheMeasure }
      :=
      \Lim[\TheLevel\rightarrow\infty]{
        \PairingOf{ \FnPath[\TheLevel] }{ \TheMeasure }
      }
    \]
    and note that the map
    \[
      \PairingOf{ \DummyArg }{ \TheMeasure }
      \mapcolon
      \LipMaps[\MetricParameter] \longrightarrow \ComplexNumbers
    \]
    is linear. As the projection operator $\TheProjection[\AltLevel]$
    is the identity on each $\Depends[\TheLevel]$ for
    $\TheLevel\leq\AltLevel$, the pairing restricts to $\TheMeasure$ on
    each $\Depends[\TheLevel]$. I.e., we have defined an extension.

    It remains to see that
    \(
      \PairingOf{ \DummyArg }{ \TheMeasure }
    \)
    it continuous. Note first, that
    $\PairingOf{\TheProjectionOf[1]{\DummyArg}}{\TheMeasure}$ is bounded:
    its the composition of the bounded projection operator
    $\TheProjection[1]$ with the restriction of $\TheMeasure$ to the
    finite dimensional space $\Depends[1]$.
    
    Writing the difference
    \(
      \PairingOf{ \DummyArg }{ \TheMeasure }
      -
      \PairingOf{\TheProjectionOf[1]{\DummyArg}}{\TheMeasure}
    \)
    as a telescoping series, we estimate
    \begin{equation}\label{telescoping}
      \AbsValueOf{
        \PairingOf{ \FnPath }{ \TheMeasure }
        -
        \PairingOf{ \FnPath[1] }{ \TheMeasure }
      }
      \leq
      \Sum[\TheLevel=1][\infty]{
        \AbsValueOf{
          \PairingOf{ \FnPath[\TheLevel+1] - \FnPath[\TheLevel] }{ \TheMeasure }
        }
      }
      \leq
      \InnerConstOf[\MetricParameter]{\FnPath}
      \frac{1}{1-\MetricParameter\MaxModDegree\AbsValueOf{\TheTwist}^{-1}}
      (\CardOf{\VertexSet})
      \MaxModDegree
      \Max[\TheEdge\in\EdgeSet]{
        \AbsValueOf{ \TheMeasureOf{\TheEdge} }
      }
    \end{equation}
    using the geometric series. It follows that
    \(
      \PairingOf{ \DummyArg }{ \TheMeasure }
    \)
    is continuous.

    Now consider an eigenform
    \(
      \TheEigenForm\in
      \Equalizer{\DualTf[\MetricParameter]}{\TheTwist\Identity}
      \subspace
      \ContDual{\LipMaps[\MetricParameter]}
    \)
    and its restriction $\TheMeasure$ to $\Depends[\infty]$. We claim
    that $\TheEigenForm$ coincides with the extension
    $\PairingOf{\DummyArg}{\TheMeasure}$ just constructed. So,
    we consider an arbitrary function $\FnPath\in\LipMaps[\MetricParameter]$
    and compute
    \[
      (\TheTwist^{\TheLevel}\TheEigenForm)(\FnPath-\TheProjection[\TheLevel]\FnPath)
      =
      ({\DualTf[\MetricParameter]}^{\TheLevel}\TheEigenForm)
      (\FnPath-\TheProjection[\TheLevel]\FnPath)
      =
      \TheEigenForm(
      \AlgTf^{\TheLevel}\FnPath-\AlgTf^{\TheLevel}\TheProjection[\TheLevel]\FnPath
      ).
    \]
    As the spectral radius of $\AlgTf$ is at most $\MaxModDegree$,
    Corollary~\ref{iteratedTf} implies that
    \[
      \AbsValueOf{
        \TheEigenForm
        (\FnPath)
        -
        \TheEigenForm(\TheProjection[\TheLevel]\FnPath)
      }
      \leq
      \NormOf[\op]{\TheEigenForm}
      \AbsValueOf{\TheTwist}^{-\TheLevel}
      (\MaxModDegree+\TheEps)^{\TheLevel}
      \MetricParameter^{\TheLevel}
      \InnerConstOf[\MetricParameter]{\FnPath}
    \]
    holds for any $\TheEps>0$, whence
    \(
      \PairingOf{\TheProjection[\TheLevel]\FnPath}{\TheMeasure}
      =
      \TheEigenForm
      (\TheProjection[\TheLevel]\FnPath)
    \)
    converges to $\TheEigenForm(\FnPath)$.

    Finally, let us assume that the projection operators
    $\TheProjection[\TheLevel]$ are chosen as in
    Observation~\ref{good-choice}. Then, the extension
    \(
      \PairingOf{ \DummyArg }{ \TheMeasure }
    \)
    is an eigenfunction of the dual transfer operator
    $\DualTf[\MetricParameter]$ for the eigenvalue $\TheTwist$. To see
    this, we compute
    \[
      \PairingOf{ \AlgTf\FnPath }{ \TheMeasure }
      =
      \Lim[\TheLevel\rightarrow\infty]{
        \PairingOf{ \TheProjectionOf[\TheLevel]{ \AlgTf\FnPath } }{ \TheMeasure }
      }
      =
      \Lim[\TheLevel\rightarrow\infty]{
        \PairingOf{ \AlgTfOf{ \TheProjection[\TheLevel+1] \FnPath } }{ \TheMeasure }
      }
      =
      \Lim[\TheLevel\rightarrow\infty]{
        \PairingOf{ \TheProjection[\TheLevel+1] \FnPath }{ \TheTwist\TheMeasure }
      }
      =
      \TheTwist\PairingOf{ \FnPath }{ \TheMeasure }
    \]
    Now all claims have been proved.
  \end{proof}

  \begin{remark}
    To complete the picture, we remark that the extension
    $\PairingOf{\DummyArg}{\TheMeasure}$ does not depend on the particular
    choice of the $\Prefer{\TheCode}$. Independence can be seen as in the
    proof of \cite[Lemma~5.11]{BHW22}. This independence does not mean that the finitely
    additive measure $\TheMeasure$ has a unique continuous extension to
    $\LipMaps[\MetricParameter]$. Since $\Depends[\infty]$ is not dense in
    $\LipMaps[\MetricParameter]$, this cannot be expected. However, any
    choice of projection operators $\TheProjection[\TheLevel]$ yields the
    same extension, which is the unique extension of $\TheMeasure$ that
    lies in the eigenspace
    $\Equalizer{\DualTf[\MetricParameter]}{\TheTwist\Identity}$.
  \end{remark}

  \begin{remark}
    For given $z\neq 0$ and for $\MetricParameter$ small enough, the eigenspaces of
    $\Tf[\MetricParameter]$ and $\DualTf[\MetricParameter]$ coincide with
    the eigenspaces
    \(
      \Equalizer{\AlgTf}{\TheTwist\Identity}
      \subspace\Depends[\infty]
    \)
    and 
    \(
      \Equalizer{\DualAlgTf}{\TheTwist\Identity}
      \subspace\FinAddMeasuresOn{\PathSet}
    \),
    respectively.
    Corollary~\ref{canonical-transpose} yields therefore an isomorphism
    \[
      \TheIsom[\MetricParameter]
      \mapcolon
      \Equalizer{\Tf[\MetricParameter]}{\TheTwist\Identity}
      \longrightarrow
      \Equalizer{\DualTf[\MetricParameter]}{\TheTwist\Identity}
    \]
    and since these eigenspaces have finite dimension, the isomorphism
    $\TheIsom[\MetricParameter]$ is automatically continuous with respect
    to the given norms on either side. However, the operator norm of
    $\TheIsom[\MetricParameter]$ depends on $\MetricParameter$.

    We can use the estimate~\pref{telescoping} to bound
    $\NormOf[\op]{\TheIsom[\MetricParameter]}$. For
    $\FnPath\in\LipMaps[\MetricParameter]$ and a finitely additive measure
    $\TheMeasure\in
    \Equalizer{\DualTf[\MetricParameter]}{\TheTwist\Identity}$, we have by
    straight forward computation:
    \begin{align*}
      \AbsValueOf{
      \PairingOf{\FnPath}{\TheMeasure}
      }
      & \leq
        \NormOf[\infty]{\FnPath}
        \Max[\TheEdge\in\EdgeSet]{
        \AbsValueOf{ \TheMeasureOf{\TheEdge} }
        }
        +
        \InnerConstOf[\MetricParameter]{\FnPath}
        \frac{1}{1-\MetricParameter\MaxModDegree\AbsValueOf{\TheTwist}^{-1}}
        (\CardOf{\VertexSet})
        \MaxModDegree
        \Max[\TheEdge\in\EdgeSet]{
        \AbsValueOf{ \TheMeasureOf{\TheEdge} }
        }
      \\
      & \leq
        \InnerNormOf[\MetricParameter]{\FnPath}
        \frac{1}{1-\MetricParameter\MaxModDegree\AbsValueOf{\TheTwist}^{-1}}
        (\CardOf{\VertexSet})
        \MaxModDegree
        \Max[\TheEdge\in\EdgeSet]{
        \AbsValueOf{ \TheMeasureOf{\TheEdge} }
        }
    \end{align*}
    So the operator norm of the extension $\PairingOf{\DummyArg}{\TheMeasure}$
    to $\LipMaps[\MetricParameter]$ is bounded from above by
    \(
      \frac{1}{1-\MetricParameter\MaxModDegree\AbsValueOf{\TheTwist}^{-1}}
      (\CardOf{\VertexSet})
      \MaxModDegree
      \Max[\TheEdge\in\EdgeSet]{
        \AbsValueOf{ \TheMeasureOf{\TheEdge} }
      }
    \).

    On the other hand, the eigenfunction in
    $\LipMaps[\MetricParameter]$ that corresponds to $\TheMeasure$ by
    Corollary~\ref{canonical-transpose} is given by:
    \[
      (\TheEdge[1]\turn\TheEdge[2]\turn\cdots)
      \longmapsto
      \TheMeasureOf{\OppositeOf{\TheEdge[1]}}
    \]
    It lies in $\Depends[1]$ where
    $\InnerNormOf[\MetricParameter]{\DummyArg}$ and
    $\NormOf[\infty]{\DummyArg}$ coincide. The norm is therefore given by
    \(   
      \Max[\TheEdge\in\EdgeSet]{
        \AbsValueOf{ \TheMeasureOf{\TheEdge} }
      }
    \).

    Hence, the operator norm of $\TheIsom[\MetricParameter]$ is
    bounded from above by
    \(
      \frac{1}{1-\MetricParameter\MaxModDegree\AbsValueOf{\TheTwist}^{-1}}
      (\CardOf{\VertexSet})
      \MaxModDegree
    \)
    which has a pole at
    $\MetricParameter=\frac{\AbsValueOf{\TheTwist}}{\MaxModDegree}$.
  \end{remark}

  We can characterize when the pairing
  $\PairingOf{\DummyArg}{\DummyArg}$ defined in~\pref{pairing} is
  degenerate on
  \(
    \Equalizer{\AlgTf}{\TheTwist\Identity}
    \crossprod
    \Equalizer{\DualAlgTf}{\TheTwist\Identity}
  \)
  by combining the isomorphisms of
  Corollary~\ref{transfer-on-loc-const} and Theorem~\ref{dual-main}.  We
  start with a lemma from linear algebra for which we were not able to
  find a reference in the standard literature.  It seems suitable as a
  textbook exercise.

  \begin{lemma}\label{degeneracy-on-eigenspaces}
    Assume that
    \(
      \TheFormOf{\DummyArg}{\DummyArg} \mapcolon
      \TheVectorSpace\crossprod\AltVectorSpace
      \rightarrow \ComplexNumbers
    \)
    is a non-degenerate bilinear pairing of two finite dimensional vector spaces.
    Assume that $\TheLinOp\mapcolon\TheVectorSpace\rightarrow\TheVectorSpace$
    and $\AltLinOp\mapcolon\AltVectorSpace\rightarrow\AltVectorSpace$ are
    adjoint operators, i.e.,
    \(
      \TheFormOf{\TheLinOp\TheVector}{\AltVector}
      =
      \TheFormOf{\TheVector}{\AltLinOp\AltVector}
    \).
    
    Then for any eigenvalue $\TheEigenvalue$, the restriction of
    \(
      \TheFormOf{\DummyArg}{\DummyArg}
    \)
    to the generalized eigenspaces
    \(
      \Ker(\TheLinOp-\TheEigenvalue)^{\TheNilpotency}
      \crossprod
      \Ker(\AltLinOp-\TheEigenvalue)^{\AltNilpotency}
    \)
    is non-degenerate. The restriction to the eigenspaces
    \(
      \Ker(\TheLinOp-\TheEigenvalue)
      \crossprod
      \Ker(\AltLinOp-\TheEigenvalue)
    \)
    is non-degenerate if and only if there are no generalized eigenvectors
    of rank~2.

    Moreover, if $\TheLinOp$ has generalized eigenvectors of rank~2 then so
    does $\AltLinOp$, and vice versa.
  \end{lemma}
  
  \begin{proof}
    In view of the Jordan normal form, there are decompositions
    \[
      \TheVectorSpace=
      \underbrace{
        \Ker(\TheLinOp-\TheEigenvalue)^{\TheNilpotency}
      }_{=:\TheGenEigenspace[\TheEigenvalue]}\directsum\,\TheComplement
      \quad\text{and}\quad
      \AltVectorSpace=
      \underbrace{
        \Ker(\AltLinOp-\TheEigenvalue)^{\AltNilpotency}
      }_{=:\AltGenEigenspace[\TheEigenvalue]} \directsum\,\AltComplement
    \]
    of $\TheVectorSpace$ and $\AltVectorSpace$ into generalized
    eigenspaces and complements such that $\TheLinOp-\TheEigenvalue$ is
    invertible on $\TheComplement$ and likewise
    $\AltLinOp-\TheEigenvalue$ is invertible on $\AltComplement$. Note
    that for $\TheVector\in\TheComplement$ and
    $\AltVector\in\AltGenEigenspace[\TheEigenvalue]$ we have:
    \[
      \TheFormOf{\TheVector}{\AltVector}
      =
      \TheFormOf{
        (\TheLinOp-\TheEigenvalue)^{\AltNilpotency}
        (\TheLinOp-\TheEigenvalue)^{-\AltNilpotency}
        \TheVector
      }{
        \AltVector
      }
      =
      \TheFormOf{
        (\TheLinOp-\TheEigenvalue)^{-\AltNilpotency}
        \TheVector
      }{
        (\AltLinOp-\TheEigenvalue)^{\AltNilpotency}
        \AltVector
      }
      =
      0
    \]
    Thus $\TheFormOf{\DummyArg}{\DummyArg}$ vanishes on
    \(
      \TheComplement\crossprod\AltGenEigenspace[\TheEigenvalue]
    \)
    and, analogously, on
    \(
      \TheGenEigenspace[\TheEigenvalue]\crossprod\AltComplement
    \).
    Hence, $\TheFormOf{\DummyArg}{\DummyArg}$ is non-degenerate on
    \(
      \TheGenEigenspace[\TheEigenvalue]\crossprod\AltGenEigenspace[\TheEigenvalue]
    \).
    
    Now, consider a rank~2 eigenvector $\TheVector$. From
    \[
      \TheFormOf{(\TheLinOp-\TheEigenvalue)\TheVector}{\AltVector}
      =
      \TheFormOf{\TheVector}{(\AltLinOp-\TheEigenvalue)\AltVector}
    \]
    we see that the eigenvector
    $(\TheLinOp-\TheEigenvalue)\TheVector$ is orthogonal to all of
    $\KernelOf{\AltLinOp-\TheEigenvalue}$. Thus, the restriction of
    $\TheFormOf{\DummyArg}{\DummyArg}$ to
    \(
      \KernelOf{\TheLinOp-\TheEigenvalue}\crossprod\KernelOf{\AltLinOp-\TheEigenvalue}
    \)
    is degenerate if $\TheLinOp$ has generalized eigenvectors of
    rank~2. In this case, $\AltGenEigenspace[\TheEigenvalue]$ is also
    strictly larger than the eigenspace
    $\KernelOf{\AltLinOp-\TheEigenvalue}$ because there is some vector
    $\AltVector\in\AltGenEigenspace[\TheEigenvalue]$ with
    $\TheFormOf{(\TheLinOp-\TheEigenvalue)\TheVector}{\AltVector}\neq 0$.
  \end{proof}
  
  \begin{proposition}\label{degeneracy}
    Let $z\not= 0$. Then the pairing
    $\PairingOf{\DummyArg}{\DummyArg}$ defined in~\pref{pairing} is
    non-degenerate when restricted to
    \(
      \Equalizer{\AlgTf}{\TheTwist\Identity}
      \crossprod
      \Equalizer{\DualAlgTf}{\TheTwist\Identity}
    \)
    if and only if the eigenspace
    $\Equalizer{\TurnSum}{\TheTwist\Identity}$ coincides with the
    generalized eigenspace of $\TurnSum$ for the eigenvalue $\TheTwist$.
  \end{proposition}
  \begin{proof}
    Consider the symmetric bilinear form
    \[
      \TheFormOf{\TheFnEdge}{\AltFnEdge}
      :=
      \Sum[\TheEdge]{
        \TheFnEdgeOf{\TheEdge}
        \AltFnEdgeOf{\OppositeOf{\TheEdge}}
      }
    \]
    on $\MapsOn{\EdgeSet}$. The basis elements in $\EdgeSet$ come in
    pairs of opposite edges, whence the matrix of
    $\TheFormOf{\DummyArg}{\DummyArg}$ has block diagonal form with
    $(\begin{smallmatrix}0 & 1 \\ 1 & 0\end{smallmatrix})$ blocks. The
    determinant clearly does not vanish, whence the form is non-degenerate
    on $\MapsOn{\EdgeSet}$.

    It is easy to verify the identity
    \(
      \TheFormOf{\TurnSum\TheFnEdge}{\AltFnEdge}
      =
      \TheFormOf{\TheFnEdge}{\TurnSum\AltFnEdge}
      ,
    \)
    i.e., that $\TurnSum$ is a self-adjoint operator on
    $\MapsOn{\EdgeSet}$.

    We apply Lemma~\ref{degeneracy-on-eigenspaces} to the given pairing with
    $\TheVectorSpace=\AltVectorSpace=\MapsOn{\EdgeSet}$ and
    $\TheLinOp=\AltLinOp=\TurnSum$. Thus, it remains to show that
    $\PairingOf{\DummyArg}{\DummyArg}$ defined in~\pref{pairing} is
    non-degenerate when restricted to
    \(
      \Equalizer{\AlgTf}{\TheTwist\Identity}
      \crossprod
      \Equalizer{\DualAlgTf}{\TheTwist\Identity}
    \)
    if and only if $\TheFormOf{\DummyArg}{\DummyArg}$ is
    non-degenerate on
    \(
      \Equalizer{\TurnSum}{\TheTwist\Identity}
      \crossprod
      \Equalizer{\TurnSum}{\TheTwist\Identity}
    \).
    This, in turn, follows from
    Corollary~\ref{transfer-on-loc-const} and Theorem~\ref{dual-main}.
    The isomorphisms are explicit, and we have the following commutative
    diagram:
    \[
      \begin{tikzcd}
        \Equalizer{\AlgTf}{\TheTwist\Identity}
        \crossprod
        \Equalizer{\DualAlgTf}{\TheTwist\Identity}
        \arrow[rd,"\PairingOf{\DummyArg}{\DummyArg}"]
        \arrow[dd,"\isomorphic"]
        &
        {}
        \\
        {}
        &
        \ComplexNumbers
        \\
        \Equalizer{\TurnSum}{\TheTwist\Identity}
        \crossprod
        \Equalizer{\TurnSum}{\TheTwist\Identity}
        \arrow[ur,"\TheFormOf{\DummyArg}{\DummyArg}"']
        &
        {}
      \end{tikzcd}
    \]
    This is to say that the isomorphism of
    Corollary~\ref{transfer-on-loc-const} identifies $\AlgTf$ with
    $\TurnSum$ and the isomorphism of Theorem~\ref{dual-main} identifies
    $\DualAlgTf$ also with $\TurnSum$. The bilinear form
    $\TheFormOf{\DummyArg}{\DummyArg}$ is defined accordingly in view of
    \pref{pairing}. As $\AlgTf$ and $\DualAlgTf$ are adjoint with respect
    to $\PairingOf{\DummyArg}{\DummyArg}$, it makes sense that $\TurnSum$
    is self-adjoint with respect to the bilinear form
    $\TheFormOf{\DummyArg}{\DummyArg}$. The claim follows.
  \end{proof}

  \begin{remark}
    In \cite[\S~3]{LP16} our turn sum operator $\TurnSum$ appears as
    \emph{nonbacktracking walk matrix}. For $\ModDegree+1$-regular graphs
    with $\ModDegree\geq 2$ \cite[Prop.~3.1]{LP16} gives a normal form of
    this operator from which one can derive that if $\pm 2\sqrt\ModDegree$
    is an eigenvalue of $\TurnSum$, then it has a nontrivial Jordan
    block. So, while \cite{LP16} does not give an example of such a graph
    we expect that there actually exist graphs for which the pairing
    $\PairingOf{\DummyArg}{\DummyArg}$ is degenerate on
    \[
      \Equalizer{\AlgTf}{\TheTwist\Identity}
      \crossprod
      \Equalizer{\DualAlgTf}{\TheTwist\Identity}.
    \]   
  \end{remark}

  \begin{remark}\label{rem:Lubetzky-Peres}
    Combining Remark~\ref{rem:correspondence} with
    Theorem~\ref{dual-main} we see that the Ruelle resonances of the
    transfer operators actually agree with the eigenvalues of
    $\DualAlgTf$. Keeping in mind the analogy between hyperfunctions on
    the boundary of symmetric space and finitely additive measures for
    trees exhibited in \cite{BHW22} we follow the terminology for
    symmetric spaces and call the elements of
    $\Equalizer{\DualAlgTf}{\TheTwist\Identity}$ \emph{resonant states}.
  \end{remark}

  \section{Review: harmonic analysis on trees}
  Another approach to understand the spectral theory of the path space
  $\PathSet$ of a connected finite graph $\TheGraph$ without end points
  is to look at its universal cover. That is a tree of bounded degree
  also without terminal vertices. In this section we review some central
  results from \cite{BHW22} in that setting. In particular, we introduce the
  Poisson transform connecting Laplace eigenfunctions to certain
  distributions on the boundary.

  So let $\TheTree$ be a tree without terminal vertices
  (i.e. $\ModDegreeOf{\TheTreeVertex}\neq 0$ for any vertex
  $\TheTreeVertex\in\TheTree$) and of bounded degree (i.e. there is a
  constant $\MaxModDegree\geq\ModDegreeOf{\TheTreeVertex}$ for any vertex
  $\TheTreeVertex\in\TheTree$).

  As we ultimately need to relate a graph $\TheGraph$ to its universal
  cover, we need to distinguish entities upstairs (in the cover) and
  downstairs. For this reason, $\TreeVertexSet$, $\TreeEdgeSet$ and
  $\TreePathSet$ refer to the sets of vertices, oriented edges, and edge
  paths without backtracking in $\TheTree$, respectively; and
  $\TreeDistance[\MetricParameter]$ denotes the
  $\MetricParameter$-distance on $\TreePathSet$. We denote the
  corresponding operators by $\TreeGradient[\TheTwist]$,
  $\TreeRescale[\TheTwist]$, $\TreeLocalTweak[\TheTwist]$,
  $\TreeTurnSum$, $\TreeNeighborSum$, $\TreeNeighborAvg$, $\TreeTf$ and
  $\TreeTf[\MetricParameter]$, and their duals $\TreeDualTf$ and
  $\TreeDualTf[\MetricParameter]$.

  In a tree, it is natural to think of an edge path without
  backtracking as a geodesic ray. Two such rays are \notion{confluent}
  if they share infinitely many adges. Confluence is an equivalence
  relation on the set $\TreePathSet$ of geodesic rays; and the set of
  equivalence classes (\notion{ends}) $\TreeBoundary$ is the
  \notion{boundary} of $\TheTree$. For an oriented edge $\TreeEdge$, the
  geodesic rays starting with $\TreeEdge$ cut out the \notion{forward
    subset} $\ForwardBoundaryOf{\TreeEdge}$ of $\TreeBoundary$. Note that
  $\TreeBoundary$ is the disjoint union of
  $\ForwardBoundaryOf{\TreeEdge}$ and
  $\ForwardBoundaryOf{\OppositeOf{\TreeEdge}}$. Declaring the forward
  subsets to be basic open sets defines a topology on
  $\TreeBoundary$. In the language of \CatZero{} spaces, $\TreeBoundary$
  with this topology is the \notion{visual boundary} of $\TheTree$. As
  $\TheTree$ is locally finite, the boundary $\TreeBoundary$ is a
  compact totally disconnected space.

  \begin{remark}
    As the terminology ``boundary'' suggests, $\TreeBoundary$ occurs
    in a compactification of $\TheTree$. Specifically, the set
    $\TreeBoundary$ coincides with the set of ends of the tree $\TheTree$
    and $\TheTree\union\TreeBoundary$ is the end compactification of
    $\TheTree$.  A basis of open sets is again given by forward sets of
    oriented edges, but this time the forward set of $\TreeEdge$ contains
    all vertices in front of $\TreeEdge$ as well as
    $\ForwardBoundaryOf{\TreeEdge}$. The end topology on
    $\TheTree\union\TreeBoundary$ induces the discrete topology on the
    vertex set of $\TheTree$ and on $\TreeBoundary$ the already described
    topology.
  \end{remark}

  For each end $\TheEnd\in\TreeBoundary$ and each vertex
  $\TheTreeVertex\in\TheTree$, there is a unique geodesic ray from
  $\TheTreeVertex$ to $\TheEnd$, for which we use the notation
  $\RayFromTo{\TheTreeVertex}{\TheEnd}$. The space of ends $\TreeBoundary$
  can thus be identified with the island of $\TreePathSet$ consisting of
  edge paths issuing from $\TheTreeVertex$. Then
  \begin{align*}
    \TreeDistance[\MetricParameter,\TheTreeVertex] \mapcolon
    \TreeBoundary \crossprod \TreeBoundary & \longrightarrow \RealNumbers \\
    (\TheEnd,\AltEnd) & \longmapsto
                        \TreeDistanceOf[\MetricParameter]
                        {\RayFromTo{\TheTreeVertex}{\TheEnd}}
                        {\RayFromTo{\TheTreeVertex}{\AltEnd}}
  \end{align*}    
  is an ultrametric on $\TreeBoundary$.
  \begin{figure}
    \begingroup
    \setlength{\unitlength}{1cm}
    \newcommand{\makeSubtree}[4]{%
      \IfThenElse{#1=0}{}{%
        \begingroup
        \pgfmathtruncatemacro{\NewDepth}{#1-1}%
        \pgfmathsetmacro{\FirstAngle}{#3}%
        \pgfmathsetmacro{\SecondAngle}{#4}%
        \pgfmathsetmacro{\Middle}{((#3)+(#4))/2}%
        \coordinate (#2L) at ($(#2)+(\FirstAngle:0.9)$);
        \coordinate (#2R) at ($(#2)+(\SecondAngle:0.9)$);
        \draw [thin, gray] (#2) -- (#2L);
        \draw [thin, gray] (#2) -- (#2R);
        \makeSubtree{\NewDepth}{#2L}{#3}{((#3)+(#4))/2}%
        \makeSubtree{\NewDepth}{#2R}{((#3)+(#4))/2}{#4}%
        \fill (#2) circle (2pt);  
        \endgroup
      }%
    }%
    \begin{center}
      \begin{tikzpicture}[scale=0.7]
        \coordinate (O) at (0:0);
        \fill (O) circle (1pt);
        \coordinate (OA) at (60:1);
        \coordinate (OB) at (180:1);
        \coordinate (OC) at (300:1);
        \draw [thin, gray] (O) -- (OA);
        \draw [thin, gray] (O) -- (OB);
        \draw [thin, gray] (O) -- (OC);
        \makeSubtree{5}{OA}{0}{120};
        \makeSubtree{5}{OB}{120}{240};
        \makeSubtree{5}{OC}{240}{360};
        \draw [very thick, black] (OBRR) -- (OBR) -- (OB) -- (O) -- (OA) -- (OAR) -- (OARR) -- (OARRL) -- (OARRLR) -- (OARRLRR);
        \draw [very thick, black] (OBRR) -- (OBR) -- (OB) -- (O) -- (OA) -- (OAL) -- (OALR) -- (OALRL) -- (OALRLR) -- (OALRLRL);
        \fill [black] (OBRR) circle (4pt);
        \fill [black] (OA) circle (4pt);
        \draw [very thick, black] (OBRR) -- (OBR) -- (OB) -- (O) -- (OA);
        \node at (OARRLRR) [anchor=south] {$\TheEnd$};
        \node at (OALRLRL) [anchor=south west] {$\AltEnd$};
        \node at (OBRR) [anchor=north west] {$\TheTreeVertex$};
      \end{tikzpicture}
    \end{center}
    \endgroup
    \caption{The geometric meaning of
      $\HoroDistOf[\TheTreeVertex]{\TheEnd}{\AltEnd}$: in the instance
      shown, the rays $\RayFromTo{\TheTreeVertex}{\TheEnd}$ and
      $\RayFromTo{\TheTreeVertex}{\AltEnd}$ share four edges, whence
      $\HoroDistOf[\TheTreeVertex]{\TheEnd}{\AltEnd}=4$.}
  \end{figure}
  In the tree case one has the
  following geometric illustration. Any two distinct ends $\TheEnd$ and
  $\AltEnd$ are joined by a unique bi-infinite geodesic in
  $\TheTree$. The union of the two rays
  $\RayFromTo{\TheTreeVertex}{\TheEnd}$ and
  $\RayFromTo{\TheTreeVertex}{\AltEnd}$ is a tripod containing the geodesic
  from $\TheEnd$ to $\AltEnd$ as its two infinite legs. The third leg is
  of finite length, which we denote by
  $\HoroDistOf[\TheTreeVertex]{\TheEnd}{\AltEnd}$. Then:
  \[
    \TreeDistanceOf[\MetricParameter,\TheTreeVertex] {\TheEnd}{\AltEnd} =
    \MetricParameter^{\HoroDistOf[\TheTreeVertex]{\TheEnd}{\AltEnd}}
  \]
  Of course, we can recover the metric that we already defined on $\TreePathSet$
  with this identification as follows:
  \[
    \TreeDistanceOf[\MetricParameter]
    {\TheTreeVertex\to\TheEnd}{\AltTreeVertex\to\AltEnd}
    =
    \begin{cases}
      1 & \TheTreeVertex \neq \AltTreeVertex
      \\
      \TreeDistanceOf[\MetricParameter,\TheTreeVertex]{\TheEnd}{\AltEnd}
        & \TheTreeVertex = \AltTreeVertex
    \end{cases}
  \]

  It is easy to see that
  $\TreeDistance[\MetricParameter,\TheTreeVertex]$ and
  $\TreeDistance[\MetricParameter,\AltTreeVertex]$ are bi-Lipschitz
  equivalent for neighboring vertices $\TheTreeVertex$ and
  $\AltTreeVertex$. It follows that the equivalence class of the metric
  $\TreeDistance[\MetricParameter,\TheTreeVertex]$ does not depend on
  the vertex $\TheTreeVertex$. Moreover, the topology defined by
  $\TreeDistance[\MetricParameter,\TheTreeVertex]$ on $\TreeBoundary$ is
  also independent of $\MetricParameter$. In fact, it coincides with the
  topology we already described above.

  Therefore, the space $\BoundedMaps{\TreeBoundary}$ of continuous
  (and automatically bounded) functions on $\TreeBoundary$ does not
  depend on $\MetricParameter$. However, which of those functions are
  Lipschitz continuous with respect to
  $\TreeDistance[\MetricParameter,\TheTreeVertex]$ does depend on
  $\MetricParameter$, although not on the vertex $\TheTreeVertex$.  We
  denote by $\LipMapsOn[\MetricParameter]{\TreeBoundary}$ the space of
  $\TreeDistance[\MetricParameter,\TheTreeVertex]$-Lipschitz continuous
  functions on $\TreeBoundary$. Note that for $\MetricParameter \leq
  \AltMetricParameter$, the metric
  $\TreeDistance[\MetricParameter,\TheTreeVertex]$ is bounded from above by
  $\TreeDistance[\AltMetricParameter,\TheTreeVertex]$. It follows that
  \(
    \LipMapsOn[\MetricParameter]{\TreeBoundary}
    \subspace
    \LipMapsOn[\AltMetricParameter]{\TreeBoundary}
  \).

  For $\FnEnds\in\LipMapsOn[\MetricParameter]{\TreeBoundary}$,
  let $\InfLipConstOf[\MetricParameter,\TheTreeVertex]{\FnEnds}$
  denote the optimal Lipschitz constant for $\FnEnds$. Note that
  $\InfLipConstOf[\MetricParameter,\TheTreeVertex]{\cdot}$ is a
  semi-norm on $\LipMapsOn[\MetricParameter]{\TreeBoundary}$. We denote
  by $\LipMapsOn[\MetricParameter,\TheTreeVertex]{\TreeBoundary}$ the
  space $\LipMapsOn[\MetricParameter]{\TreeBoundary}$ endowed with the
  norm
  \[
    \LipNormOf[\MetricParameter,\TheTreeVertex]{\FnEnds}
    :=
    \InfLipConstOf[\MetricParameter,\TheTreeVertex]{\FnEnds}
    +
    \NormOf[\infty]{\FnEnds}
  \]
  obtained by adding the sup-norm (recall that $\FnEnds$ is bounded).
  We use the notation $\LipMapsOn[\MetricParameter]{\TreeBoundary}$ to talk
  about the topological vector space without choosing a particular norm.

  For $\MetricParameter\leq\AltMetricParameter$, we observe
  \(
    \InfLipConstOf[\MetricParameter,\TheTreeVertex]{\FnEnds}
    \leq
    \InfLipConstOf[\AltMetricParameter,\TheTreeVertex]{\FnEnds}
  \)
  whence the inclusion
  $\LipMapsOn[\MetricParameter,\TheTreeVertex]{\TreeBoundary}$ into
  $\LipMapsOn[\AltMetricParameter,\TheTreeVertex]{\TreeBoundary}$ is
  continuous with operator norm~$1$.

  For a given end $\TheEnd$ and two vertices $\TheTreeVertex$ and
  $\AltTreeVertex$, the geodesic rays
  $\RayFromTo{\TheTreeVertex}{\TheEnd}$ and
  $\RayFromTo{\AltTreeVertex}{\TheEnd}$ share infinitely many
  vertices. The difference
  \[
    \HoroDistOf[\TheEnd]{\TheTreeVertex}{\AltTreeVertex}
    :=
    \VertexDistance{\TheTreeVertex}{\TreeBaseVertex}
    -
    \VertexDistance{\AltTreeVertex}{\TreeBaseVertex}
  \]
  of vertex distances in $\TheTree$ is independent of which vertex
  \(
    \TreeBaseVertex
  \)
  is chosen in the intersection
  \(
    (\RayFromTo{\TheTreeVertex}{\TheEnd})
    \intersect
    (\RayFromTo{\AltTreeVertex}{\TheEnd})
  \).
  For fixed $\TheTreeVertex$ and $\AltTreeVertex$, the function
  \[
    \TheEnd \longmapsto \HoroDistOf[\TheEnd]{\TheTreeVertex}{\AltTreeVertex}
  \]
  is a locally constant function on $\TreeBoundary$.

  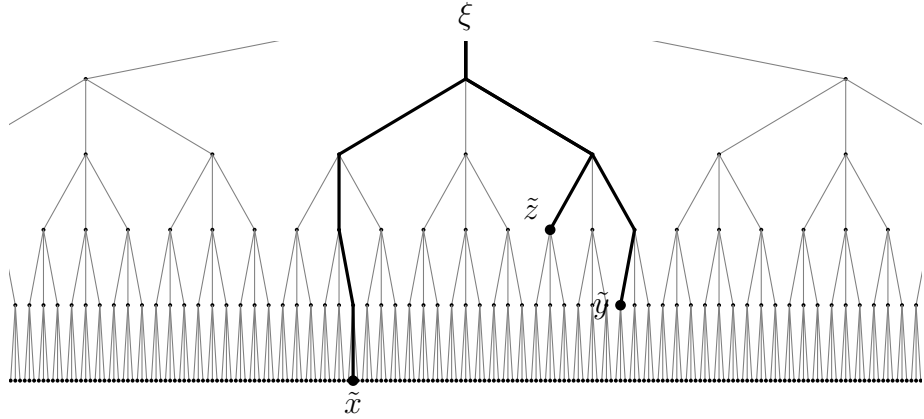
\begin{figure}
    \begingroup
    \newcommand{\makeSubtree}[3]{%
      \IfThenElse{#1=0}{}{%
        \begingroup
        \pgfmathtruncatemacro{\Next}{#1-1}%
        \pgfmathsetmacro{\TheWidth}{#3}%
        \coordinate (#2D) at ($(#2)+(0,-1)$);
        \coordinate (#2L) at ($(#2)+(\TheWidth,-1)$);
        \coordinate (#2R) at ($(#2)+(-\TheWidth,-1)$);
        \draw [thin, gray] (#2) -- (#2D);
        \draw [thin, gray] (#2) -- (#2L);
        \draw [thin, gray] (#2) -- (#2R);
        \fill (#2D) circle (0.75pt);
        \fill (#2L) circle (0.75pt);
        \fill (#2R) circle (0.75pt);
        \makeSubtree{\Next}{#2D}{((#3)/3)}%
        \makeSubtree{\Next}{#2L}{((#3)/3)}%
        \makeSubtree{\Next}{#2R}{((#3)/3)}%
        \endgroup
      }%
    }%
    \begin{center}
      \begin{tikzpicture}
        \coordinate (O) at (0,1);
        \node at (0,0.5) [anchor=south] {$\TheEnd$};
        \clip (-6,0.5) rectangle (6,-4.5);
        \makeSubtree{5}{O}{5};
        \draw [very thick, black] (O) -- (OD) -- (ODL) -- (ODLR);
        \draw [very thick, black] (O) -- (OD) -- (ODL) -- (ODLL) -- (ODLLR);
        \draw [very thick, black] (O) -- (OD) -- (ODR) -- (ODRD) -- (ODRDL) -- (ODRDLD);
        \fill (ODRDLD) circle (2pt);
        \fill (ODLR) circle (2pt);
        \fill (ODLLR) circle (2pt);
        \node at (ODRDLD) [anchor=north] {$\TheTreeVertex$};
        \node at (ODLR) [anchor=south east] {$\ThrTreeVertex$};
        \node at (ODLLR) [anchor=east] {$\AltTreeVertex$};
      \end{tikzpicture}
    \end{center}
    \endgroup
    \caption{Horospheres as level sets: the geometric meaning of
      equation~\pref{cocycle-raw}}
  \end{figure}
  If $\TheTree$ is drawn so that the end $\TheEnd$ goes north and all
  other ends dangle down south, the vertices of $\TheTree$ naturally
  arrange in levels and
  $\HoroDistOf[\TheEnd]{\TheTreeVertex}{\AltTreeVertex}$ measures the
  oriented difference of levels for $\TheTreeVertex$ and
  $\AltTreeVertex$. From this interpretation, the following
  three-vertex-identity is obvious:
  \begin{equation}\label{cocycle-raw}
    \HoroDistOf[\TheEnd]{\TheTreeVertex}{\ThrTreeVertex}
    =
    \HoroDistOf[\TheEnd]{\TheTreeVertex}{\AltTreeVertex}
    +
    \HoroDistOf[\TheEnd]{\AltTreeVertex}{\ThrTreeVertex}
  \end{equation}
  Fixing a base vertex $\TreeBaseVertex$ will allow us to label the levels
  by integers and assign $0$ to the level of the base vertex.

  \begin{remark}
    In~\cite{BHW22}, the base point never changes. In that case,
    it is convenient to define the \notion{horocycle bracket}. The
    translation between the convention used in~\cite{BHW22} and this paper
    is given as follows:
    \begin{equation}\label{def-horocycle-braket}
      \HoroBraket{\TheTreeVertex}{\TheEnd}
      :=
      \HoroDistOf[\TheEnd]{\TreeBaseVertex}{\TheTreeVertex}
    \end{equation}
  \end{remark}

  We then define for any given end $\TheEnd$ and any given parameter
  $\TheTwist\in\ComplexNumbers$ the function
  \begin{align*}
    \TheEigenFct[\TheTwist,\TheEnd] \mapcolon \TreeVertexSet
    & \longrightarrow \ComplexNumbers \\
    \TheTreeVertex
    & \longmapsto \TheTwist^{\HoroDistOf[\TheEnd]{\TreeBaseVertex}{\TheTreeVertex}}
      =\TheTwist^{\HoroBraket{\TheTreeVertex}{\TheEnd}}
  \end{align*}
  which lies in the equalizer
  $\Equalizer{\TreeNeighborSum}{\TreeRescale[\TheTwist]}$. Moving the
  base vertex only changes $\TheEigenFct[\TheTwist,\TheEnd]$ up to a
  scalar multiple.  Truly new solutions to the equalizer equation can be
  obtained by linear combinations of the fundamental solutions for
  different ends. For fixed $\TheTreeVertex$, the function
  \(
    \TheEnd\mapsto\TheEigenFctOf[\TheTwist,\TheEnd]{\TheTreeVertex}
  \)
  is locally constant, whence the integral
  \[
    \TheEigenFct[\TheTwist,\TreeMeasure] \mapcolon
    \TheTreeVertex
    \longmapsto
    \Int[\TheEnd\in\TreeBoundary]{
      \TheEigenFctOf[\TheTwist,\TheEnd]{\TheTreeVertex}
      \diff\TreeMeasureOf{\TheEnd}
    }
  \]
  is well defined for any finitely additive measure $\TreeMeasure$ on
  $\TreeBoundary$, see \cite[Proposition~3.7]{BHW22}. The
  \notion{Poisson transform}
  \begin{align*}
    \PoissonTf[\TheTwist]
    \mapcolon
    \FAMeasuresOn{\TreeBoundary}
    & \longrightarrow
      \Equalizer{\TreeNeighborSum}{\Rescale[\TheTwist]}
      =
      \Equalizer{\TreeNeighborAvg}{\LocalTweak[\TheTwist]}
      \subseteq
      \MapsOn{\TreeVertexSet}
    \\
    \TreeMeasure
    & \longmapsto
      \TheEigenFct[\TheTwist,\TreeMeasure]
  \end{align*}
  defines a linear map from the space of complex-valued finitely
  additive measures on $\TreeBoundary$, which is also the algebraic dual
  of the space of locally constant functions on $\TreeBoundary$, to the
  equalizer $\Equalizer{\TreeNeighborAvg}{\TreeLocalTweak[\TheTwist]}$
  within $\MapsOn{\TreeVertexSet}$.

  Note that the Poisson transform depends on the base vertex
  $\TreeBaseVertex$, which we usually suppress in the notation. If we want
  to stress the dependency or are working with Poisson transforms based
  at different vertices, we will make the dependency explicit by writing
  $\PoissonTfAt[\TheTwist]{\TreeBaseVertex}$.

  Using the three-vertex-identity~\pref{cocycle-raw}, it is easy to
  relate Poisson transformations at different base vertices:
  \begin{align}\label{transform-transform}
    (\PoissonTfAtOf[\TheTwist]{\AltTreeVertex}{\TreeMeasure})(\TheTreeVertex)
    & =
      \Int[\TheEnd\in\TreeBoundary]{
      \TheTwist^{\HoroDistOf[\TheEnd]{\AltTreeVertex}{\TheTreeVertex}}
      \diff\TreeMeasureOf{\TheEnd}
      }
    \\
    & =
      \Int[\TheEnd\in\TreeBoundary]{
      \TheTwist^{\HoroDistOf[\TheEnd]{\TreeBaseVertex}{\TheTreeVertex}}
      \TheTwist^{\HoroDistOf[\TheEnd]{\AltTreeVertex}{\TreeBaseVertex}}
      \diff\TreeMeasureOf{\TheEnd}
      }
      =
      (
      \PoissonTfAtOf[\TheTwist]{\TreeBaseVertex}{
      \Cocycle[\TheTwist]{\AltTreeVertex}{\TreeBaseVertex}\TreeMeasure
      }
      )(\TheTreeVertex)\nonumber
  \end{align}
  where
  \(
    \Cocycle[\TheTwist]{\AltTreeVertex}{\TreeBaseVertex}
    \mapcolon\TreeBoundary\rightarrow\ComplexNumbers
  \)
  is the locally constant function
  \(
    \TheEnd \mapsto
    \TheTwist^{\HoroDistOf[\TheEnd]{\AltTreeVertex}{\TreeBaseVertex}}
    =
    \TheTwist^{-\HoroBraket{\AltTreeVertex}{\TheEnd}}    
  \).

  For $\TheTwist\not\in\SetOf{-1,0,1}$, the Poisson transform
  $\PoissonTf[\TheTwist]$ is a linear isomorphism whose inverse, the
  \notion{boundary value map}, can be described as follows.
  
  \begin{proposition}[{\cite[Theorem~4.7]{BHW22}}]
    Assume $\TheTwist\not\in\SetOf{-1,0,1}$.  For $\FnTreeVert\in
    \Equalizer{\TreeNeighborAvg}{\TreeLocalTweak[\TheTwist]}$ there is a
    unique finitely additive measure $\TreeMeasure$ such that
    $\FnTreeVert=\PoissonTfOf[\TheTwist]{\TreeMeasure}$. This measure has
    the property that on an oriented edge $\TreeEdge$ pointing away from
    the base vertex $\TreeBaseVertex$, one has:
    \begin{equation}\label{boundary}
      \TheTwist
      (\TreeGradient[\TheTwist]\FnTreeVert)(
      \OppositeOf{\TreeEdge}
      )
      =
      \TheTwist
      \FnTreeVertOf{\TerminalOf{\TreeEdge}}
      -
      \FnTreeVertOf{\InitialOf{\TreeEdge}}
      =
      (
      \TheTwist^2 - 1
      )
      \TheTwist^{\VertexDistance{\TreeBaseVertex}{\InitialOf{\TreeEdge}}}
      \TreeMeasureOf{
        \ForwardBoundaryOf{\TreeEdge}
      }
    \end{equation}
  \end{proposition}

  Any locally constant function on $\TreeBoundary$ belongs to
  $\LipMapsOn[\MetricParameter]{\TreeBoundary}$ for any parameter
  $\MetricParameter$. Consequently, any continuous linear functional
  on $\LipMapsOn[\MetricParameter]{\TreeBoundary}$ induces a unique finitely
  additive measure on $\TreeBoundary$. We therefore can regard the
  continuous dual $\ContDual{\LipMapsOn[\MetricParameter]{\TreeBoundary}}$
  as a space of finitely additive measures on $\TreeBoundary$; and for
  a measure $\TreeMeasure$, the condition
  \(
    \TreeMeasure\in\ContDual{\LipMapsOn[\MetricParameter]{\TreeBoundary}}
  \)
  says something about the regularity of $\TreeMeasure$. The higher
  $\MetricParameter$, the larger
  $\LipMapsOn[\MetricParameter]{\TreeBoundary}$ and the smaller
  $\ContDual{\LipMapsOn[\MetricParameter]{\TreeBoundary}}$. So: higher
  $\MetricParameter$ means better regularity.

  It turns out that there
  is a relation between the regularity of $\TreeMeasure$ and the growth of
  its Poisson transform
  \(
    \TheEigenFct[\TheTwist,\TreeMeasure]
    \mapcolon\TheTree\rightarrow\ComplexNumbers
  \).
  To make this precise, let
  \(
    \ShadowOf[\TreeBaseVertex]{\TheTreeVertex}
    :=
    \SetOf[\TheEnd]{
      \TheTreeVertex \in (\RayFromTo{\TreeBaseVertex}{\TheEnd})
    }
  \)
  denote the set of directions $\TheEnd$ along which $\TheTreeVertex$ is
  in the line of sight from $\TreeBaseVertex$.
  \begin{proposition}[{\cite[Lemma~5.11]{BHW22}}]
    For any measure $\TreeMeasure$ in
    $\ContDual{\LipMapsOn[\MetricParameter]{\TreeBoundary}}$ and
    $\TheBase>\frac{1}{\MetricParameter}$ there exists $\MultConst>0$ such
    that
    \begin{equation}\label{growth}
      \AbsValueOf{
        \TreeMeasureOf{
          \ShadowOf[\TreeBaseVertex]{\TheTreeVertex}
        }
      }
      \leq
      \MultConst\TheBase^{
        \VertexDistance{\TreeBaseVertex}{\TheTreeVertex}
      }
      \qquad\text{\ for all\ }
      \TheTreeVertex\in\TheTree
    \end{equation}
    Conversely, assume that condition~\pref{growth} holds for a finitely
    additive measure $\TreeMeasure$ and constants $\TheBase$ and
    $\MultConst$. Then for any $\MetricParameter \in (0,
    \frac{1}{\TheBase\MaxModDegree})$, we have
    $\TreeMeasure\in\ContDual{\LipMapsOn[\MetricParameter]{\TreeBoundary}}$,
    i.e., $\TreeMeasure$ extends to a continuous linear functional on
    $\LipMapsOn[\MetricParameter]{\TreeBoundary}$.
  \end{proposition}
  As a consequence, we can deduce some regularity for the boundary values of
  bounded functions.
  \begin{corollary}\label{cor:Poisson trafo}
    Assume $\TheTwist\not\in\SetOf{-1,0,1}$.  The Poisson transform
    $\PoissonTf[\TheTwist]$ restricts to an injective linear operator on
    each $\ContDual{\LipMapsOn[\MetricParameter]{\TreeBoundary}}$.  For
    $\MetricParameter\in(0,\frac{\AbsValueOf{\TheTwist}}{\MaxModDegree})$
    the image
    \(
      \PoissonTfOf[\TheTwist]{
        \ContDual{\LipMapsOn[\MetricParameter]{\TreeBoundary}}
      }
    \)
    contains all bounded functions in the equalizer
    \(
      \Equalizer{\TreeNeighborSum}{\TreeRescale[\TheTwist]}
    \)
    or equivalently in
    \(
      \Equalizer{\TreeNeighborAvg}{\TreeLocalTweak[\TheTwist]}
    \).
  \end{corollary}
  \begin{proof}
    The restriction of an isomorphism is automatically injective.

    Let
    \(
      \FnTreeVert \mapcolon \TreeVertexSet \rightarrow \ComplexNumbers
    \)
    be a bounded element of the equalizer
    $\Equalizer{\TreeNeighborSum}{\TreeRescale[\TheTwist]}$ and let the
    finitely additive measure $\TreeMeasure$ be its corresponding boundary
    value.  It is clear from formula~\pref{boundary} that $\TreeMeasure$
    satisfies a condition of type~\pref{growth} whenever $\TheBase >
    \AbsValueOf{\TheTwist^{-1}}$, i.e., $\TreeMeasure$ extends to a
    continuous linear functional on
    $\LipMapsOn[\MetricParameter]{\TreeBoundary}$ provided
    \(
      0 < \MetricParameter < \frac{\AbsValueOf{\TheTwist}}{\MaxModDegree}
    \).
  \end{proof}
  
  \section{Upstairs and downstairs}\label{sec:up-and-down}
  We return to a finite graph $\TheGraph$. The universal covering
  space of $\TheGraph$ now is a tree, which we denote by $\TheTree$ and
  to which we apply the harmonic analysis from the previous section.

  We briefly recall the construction of the universal cover. Fix a
  base vertex $\BaseVertex$ ``\notion{downstairs}'' in $\TheGraph$. An
  edge path (possibly empty) starting at $\BaseVertex$ corresponds to a
  vertex in the universal cover $\TheTree$ and two such paths define the
  same vertex in $\TheTree$ if they are homotopic relative to end
  points. The covering map is projection to the end point.

  In a graph, there is a one-to-one correspondence of edge paths
  without backtracking and homotopy classes of edge paths. Therefore,
  the vertices in $\TheTree$ are in one-to-one correspondence with the
  edge paths without backtracking starting at $\BaseVertex$, i.e., with
  the postal codes for the island designated by $\BaseVertex$. This way,
  we see that $\TheTree$ is one of the trees in the forest of all postal
  codes. We choose the empty word, corresponding to the root of this
  tree, as the base vertex $\TreeBaseVertex$ ``\notion{upstairs}'' in
  $\TheTree$. It maps to $\BaseVertex$ under the covering projection.

  The geodesic rays in $\PathSet[\BaseVertex]$ correspond the rays
  in $\TheTree$ starting at $\TreeBaseVertex$. Thus, the sets
  $\PathSet[\BaseVertex]$, $\TreePathSet[\TreeBaseVertex]$ and
  $\TreeBoundary$ are in bijection. Moreover, the topologies we gave on
  those sets match. Even more, the $\MetricParameter$-metrics coincide
  for each $\MetricParameter$. In particular, a finitely additive
  measure $\TheMeasure$ on $\PathSet$ restricts to
  $\PathSet[\BaseVertex]$ and can then be reinterpreted as a finitely
  additive measure $\TreeMeasure$ on $\TreeBoundary$. We call the
  induced map
  \begin{align*}
    \Restriction\mapcolon
    \FinAddMeasuresOn{\PathSet}
    &\longrightarrow
      \FinAddMeasuresOn{\TreeBoundary}
    \\
    \TheMeasure & \longmapsto \TreeMeasure
  \end{align*}
  the \notion{restriction homomorphism}. It is clearly onto, but as
  $\TheGraph$ has more than a single vertex, the kernel of the
  restriction homomorphism is huge (consisting of a copy of
  $\FinAddMeasuresOn{\TreeBoundary}$ for each vertex in $\TheGraph$ apart
  from $\BaseVertex$).

  The fundamental group $\PiOne$ of $\TheGraph$ is a free group of
  finite rank. It acts freely on the universal covering $\TheTree$ by
  deck transformations and $\TheGraph$ can be regarded as the orbit
  space $\OrbitSpace{\PiOne}{\TheTree}$. There are induced actions on
  $\TreeVertexSet$, $\TreeEdgeSet$ and $\TreePathSet$. Moreover, there
  are induced representations on the associated function spaces, which
  we shall discuss here. In particular, we shall consider the way in
  which the Poisson transform is equivariant, which requires some
  attention because of its dependency on a chosen base vertex in
  $\TheTree$.

  Let $\TreeAutom \mapcolon \TheTree \rightarrow \TheTree$ be a
  tree automorphism $\TreeAutom \in \AutOf{\TheTree}$, i.e.,
  $\TreeAutom$ consists of two permutations, one of the vertex set
  $\TreeVertexSet$ and one of the set $\TreeEdgeSet$ of directed edges
  so that both end point maps
  $\Initial\mapcolon\TreeEdgeSet\rightarrow\TreeVertexSet$ and
  $\Terminal\mapcolon\TreeEdgeSet\rightarrow\TreeVertexSet$ commute with
  $\TreeAutom$. Then we have an obvious induced permutation on the set
  $\TreePathSet$ of geodesic rays in $\TheTree$ which is compatible with
  the confluence relation of paths. Hence, we also see an induced
  permutation on the boundary $\TreeBoundary$. We call the induced
  actions of the automorphism group $\AutOf{\TheTree}$ on the sets
  $\TreeVertexSet$, $\TreeEdgeSet$, $\TreePathSet$ and $\TreeBoundary$
  the \notion{regular permutation actions}. All our group actions are
  from the left, and we shall keep it that way by occasionally paying
  the price of inserting an inverse here or there, like so: for a
  function $\FnEnds\mapcolon\TreeBoundary\rightarrow\ComplexNumbers$, we
  define
  $\TreeAutom\FnEnds\mapcolon\TreeBoundary\rightarrow\ComplexNumbers$ by
  \[
    (\TreeAutom\FnEnds)(\TheEnd)
    :=
    \FnEndsOf{\InvAutom\TheEnd}
  \]
  and since $\ExplicitMapsFromTo{\DummyArg}{\ComplexNumbers}$ is a
  contravariant functor, this defines an action of $\AutOf{\TheTree}$ on
  \(
    \ExplicitMapsFromTo{\TreeBoundary}{\ComplexNumbers}
  \)
  . In the same way we obtain an action of $\AutOf{\TheTree}$ on
  $\FAMeasuresOn{\TreeBoundary}$ by setting
  \(
    (\TreeAutom\TheMeasure)(\ForwardBoundaryOf{\TreeEdge})
    :=
    \TheMeasure(\ForwardBoundaryOf{\InvAutom\TreeEdge})
    .
  \)
  We proceed analogously for functions on $\TreeVertexSet$,
  $\TreeEdgeSet$ and $\TreePathSet$, turning all regular permutation
  actions into \notion{regular representations} of
  $\AutOf{\TheTree}$ on various function spaces including the space of
  finitely additive measures on $\TreeBoundary$.

  All these actions and representations restrict to the fundamental
  group $\PiOne$ that we regard as the group of deck transformations and
  thereby as a subgroup of $\AutOf{\TheTree}$. The covering projection
  \[
    \TheTree\longrightarrow\TheGraph
  \]
  allows us to lift functions defined on $\TheGraph$ to
  $\PiOne$-periodic functions on $\TheTree$. Conversely, every
  $\PiOne$-periodic function on $\TheTree$ induces a map on the orbit
  space $\TheGraph=\OrbitSpace{\PiOne}{\TheTree}$. Hence, we have a pull
  back isomorphism
  \[
    \ExplicitMapsFromTo{\VertexSet}{\ComplexNumbers}
    \isomorphic
    \Invariants[\PiOne]{
      \ExplicitMapsFromTo{\TreeVertexSet}{\ComplexNumbers}
    }
  \]
  and its companion isomorphisms on the other function
  spaces. Here, we denote by
  \(
    \Invariants[\PiOne]{
      \ExplicitMapsFromTo{\TreeVertexSet}{\ComplexNumbers}
    }
  \)
  the subspace of $\PiOne$-periodic functions.

  \begin{observation}\label{equiv-isom}
    The projection $\TheTree\rightarrow\TheGraph$ is a covering
    projection. The operators $\TurnSum$, $\Tf$, $\NeighborSum$,
    $\NeighborAvg$, $\Gradient[\TheTwist]$, $\Rescale[\TheTwist]$ and
    $\LocalTweak[\TheTwist]$ on $\TheGraph$ have purely local definitions
    which therefore lift to their counterparts $\TreeTurnSum$, $\TreeTf$,
    $\TreeNeighborSum$, $\TreeNeighborAvg$, $\TreeGradient[\TheTwist]$,
    $\TreeRescale[\TheTwist]$ and $\TreeLocalTweak[\TheTwist]$ on
    $\TheTree$.  Hence the canonical pull back isomorphisms restrict to
    isomorphisms
    \[
      \begin{tikzcd}
        &
        \Equalizer{
          \TurnSum
        }{
          \TheTwist\Identity
        }
        \ar[r, equal]
        \ar[d, equal, "{\Gradient[\TheTwist]}"']
        &
        \Invariants[\PiOne]{
          \Equalizer{
            \TreeTurnSum
          }{
            \TheTwist\Identity
          }
        }
        \ar[d, equal, "{\TreeGradient[\TheTwist]}"]
        &
        \\
        \Equalizer{
          \NeighborAvg
        }{
          \LocalTweak[\TheTwist]
        }
        \ar[r, equal]
        &
        \Equalizer{
          \NeighborSum
        }{
          \Rescale[\TheTwist]
        }
        \ar[r, equal]
        &
        \Invariants[\PiOne]{
          \Equalizer{
            \TreeNeighborSum
          }{
            \TreeRescale[\TheTwist]
          }
        }
        \ar[r, equal]
        &
        \Invariants[\PiOne]{
          \Equalizer{
            \TreeNeighborAvg
          }{
            \TreeLocalTweak[\TheTwist]
          }
        }
      \end{tikzcd}
    \]
    of equalizers.

    The twisted gradient $\TreeGradient[\TheTwist]$ on $\TheTree$ is
    an isomorphism of the right hand sides (provided
    $\TheTwist\not\in\SetOf{-1,0,1}$), just as the twisted gradient
    $\Gradient[\TheTwist]$ is an isomorphism for the left hand sides.
  \end{observation}

  Since the automorphism $\TreeAutom$ can move the base vertex
  $\TreeBaseVertex$ and the Poisson transform is defined using
  $\TreeBaseVertex$, we cannot expect the Poisson transform to be
  $\AutOf{\TheTree}$-equivariant. Instead, we have the following
  \notion{intertwining} identity:
  \begin{equation}\label{intertwining}
    \PoissonTfOf[\TheTwist]{
      \Cocycle[\TheTwist]{\TreeAutom\TreeBaseVertex}{\TreeBaseVertex}
      \TreeAutom\TreeMeasure
    }
    =
    \TreeAutom
    \PoissonTfOf[\TheTwist]{\TreeMeasure}
  \end{equation}
  To see this, we note that if we allow to move the base vertex, we
  are in a much simpler situation. The Poisson transform is defined in
  purely metric terms, so isometric pictures lead to identical values,
  whence
  \(
    \PoissonTfAtOf[\TheTwist]
    {\TreeAutom\TreeBaseVertex}{\TreeAutom\TreeMeasure}
    (\TreeAutom\TheTreeVertex)
    =
    \PoissonTfAtOf[\TheTwist]{\TreeBaseVertex}{\TreeMeasure}
    (\TheTreeVertex)
  \).
  Using~\pref{transform-transform}, we obtain:
  \[
    \PoissonTfAtOf[\TheTwist]{\TreeBaseVertex}{
      \Cocycle[\TheTwist]{\TreeAutom\TreeBaseVertex}{\TreeBaseVertex}
      \TreeAutom\TreeMeasure
    }(\TheTreeVertex)
    =
    \PoissonTfAtOf[\TheTwist]{\TreeBaseVertex}{\TreeMeasure}
    (\InvAutom\TheTreeVertex)
    =
    (\TreeAutom\PoissonTfAtOf[\TheTwist]
    {\TreeBaseVertex}{\TreeMeasure})(\TheTreeVertex)
  \]

  So, we define the linear representation
  \(
    \TwistedAction$ of $\AutOf{\TheTree}
  \)
  on the space of finitely additive measures $\FinAddMeasuresOn{\TreeBoundary}$:
  \[
    \TwistedActionOf[\TreeAutom]{\TreeMeasure}
    :=
    \Cocycle[\TheTwist]{\TreeAutom\TreeBaseVertex}{\TreeBaseVertex}
    \TreeAutom\TreeMeasure
  \]
  
  \begin{observation}\label{twisted-isom}
    For $\TheTwist\not\in\SetOf{-1,0,1}$ the Poisson transform
    $\PoissonTf[\TheTwist]$ defines a linear isomorphism from the space
    $\FinAddMeasuresOn{\TreeBoundary}$ to the equalizer
    $\Equalizer{\TreeNeighborSum}{\TreeRescale[\TheTwist]}$. This
    isomorphism intertwines the representation $\TwistedAction$ on
    $\FinAddMeasuresOn{\TreeBoundary}$ and the regular representation
    restricted to
    $\Equalizer{\TreeNeighborSum}{\TreeRescale[\TheTwist]}$.

    In particular, the subspaces of $\PiOne$-invariants on both sides are
    isomorphic under the Poisson transform.
  \end{observation}

  Whenever the Poisson transform is an isomorphism, it can be used to
  transpose any linear representation of $\PiOne$ on either side to the
  other. Hence, we know a priori that $\TwistedAction$ will be a linear
  representation of $\PiOne$ on $\FinAddMeasuresOn{\TreeBoundary}$. This
  provides a roundabout way to deduce the following \notion{cocycle
    property}:
  \begin{equation}\label{cocycle-identity}
    \Cocycle[\TheTwist]{
      \TreeAutom[1]\TreeAutom[2]\TreeBaseVertex
    }{\TreeBaseVertex}(\TheEnd)
    =
    \Cocycle[\TheTwist]{
      \TreeAutom[2]\TreeBaseVertex
    }{\TreeBaseVertex}(\TreeAutom[1]\TheEnd)
    \Cocycle[\TheTwist]{
      \TreeAutom[1]\TreeBaseVertex
    }{\TreeBaseVertex}(\TheEnd)
  \end{equation}
  A direct proof uses the three-vertex-identity~\pref{cocycle-raw}
  to first deduce the \notion{horocycle identity}
  \begin{equation}\label{horocycle-identity}
    \HoroBraket{\TreeAutom\TheTreeVertex}{\TreeAutom\TheEnd}
    =
    \HoroBraket{\TheTreeVertex}{\TheEnd}
    +\HoroBraket{\TreeAutom\TreeBaseVertex}{\TreeAutom\TheEnd}
  \end{equation}
  from
  \begin{align*}
    \HoroDistOf[\TreeAutom\TheEnd]{\TreeBaseVertex}{\TreeAutom\TheTreeVertex}
    &=\HoroDistOf[\TheEnd]{\TreeAutom[][-1]\TreeBaseVertex}{\TheTreeVertex}
      =\HoroDistOf[\TheEnd]{\TreeBaseVertex}{\TheTreeVertex}
      +
      \HoroDistOf[\TheEnd]{\TreeAutom[][-1]\TreeBaseVertex}{\TreeBaseVertex}
      =\HoroDistOf[\TheEnd]{\TreeBaseVertex}{\TheTreeVertex}
      +
      \HoroDistOf[\TreeAutom\TheEnd]{\TreeBaseVertex}{\TreeAutom\TreeBaseVertex}
  \end{align*}
  and then establish~\pref{cocycle-identity} by direct computation.

  Taken together, the Observations~\ref{equiv-isom}
  and~\ref{twisted-isom} imply that the Poisson transform composes with
  the pull back to yield an isomorphism of the equalizer
  $\Equalizer{\NeighborSum}{\Rescale[\TheTwist]}$ to the space of those
  finitely additive measures
  $\TreeMeasure\in\FinAddMeasuresOn{\TreeBoundary}$ that satisfy the
  invariance condition
  \(
    \TwistedActionOf[\TreeAutom]{\TreeMeasure}
    = \TreeMeasure
  \)
  for all $\TreeAutom\in\PiOne$. As this is the $\PiOne$-invariant
  subspace of $\FinAddMeasuresOn{\TreeBoundary}$ with respect to the
  representation $\TwistedAction$, we shall denote it by
  $\TwistedInvariants[\PiOne]{\FinAddMeasuresOn{\TreeBoundary}}$.

  \begin{lemma}\label{inv-equiv}
    For $\TreeMeasure\in \FinAddMeasuresOn{\TreeBoundary}$ and
    $\TreeAutom\in\AutOf{\TheTree}$, we have
    \(
      \TwistedActionOf[\TreeAutom]{\TreeMeasure}
      =
      \TreeMeasure
    \)
    if and only if
    \[
      \TheTwist^{\VertexDistance{\TreeBaseVertex}{\InitialOf{\TreeEdge}}}
      \TreeMeasureOf{\ForwardBoundaryOf{\TreeEdge}}
      =
      \TheTwist^{\VertexDistance{\TreeBaseVertex}{\InitialOf{\TreeAutom\TreeEdge}}}
      \TreeMeasureOf{\ForwardBoundaryOf{\TreeAutom\TreeEdge}}
    \]
    for each edge $\TreeEdge\in\TreeEdgeSet$ pointing away from
    $\TreeBaseVertex$ and away from $\InvAutom\TreeBaseVertex$.
  \end{lemma}
  \begin{proof}
    We start by pointing out that the function
    $\Cocycle[\TheTwist]{\TreeAutom\TreeBaseVertex}{\TreeBaseVertex}$ is
    constant on a forward boundary $\ForwardBoundaryOf{\TreeEdge}$ if the
    edge $\TreeEdge$ points away from $\TreeBaseVertex$ and
    $\InvAutom\TreeBaseVertex$. That explains the chosen restriction.
    
    Let $\TreeEdge$ be an oriented edge pointing away from
    $\TreeBaseVertex$. If
    $\VertexDistance{\InitialOf{\TreeEdge}}{\TreeBaseVertex}$ exceeds
    $\VertexDistance{\InvAutom\TreeBaseVertex}{\TreeBaseVertex}$, the edge
    $\TreeEdge$ also points away from $\InvAutom\TreeBaseVertex$. It
    follows that the forward boundaries $\ForwardBoundaryOf{\TreeEdge}$ of
    such edges cover $\TreeBoundary$. As a consequence, any finitely
    additive measure is completely determined by its values on the sets
    $\ForwardBoundaryOf{\TreeEdge}$ for edges pointing away from both,
    $\TreeBaseVertex$ and $\InvAutom\TreeBaseVertex$.

    Thus,
    \(
      \TwistedActionOf[\TreeAutom]{\TreeMeasure}
      =
      \TreeMeasure
    \)
    if and only if
    \(
      \TwistedActionOf[\TreeAutom]{\TreeMeasure}(\ForwardBoundaryOf{\TreeEdge})
      =
      \TreeMeasureOf{\ForwardBoundaryOf{\TreeEdge}}
    \)
    for each edge $\TreeEdge\in\TreeEdgeSet$ pointing away from
    $\TreeBaseVertex$ and away from $\InvAutom\TreeBaseVertex$.
    The claimed equivalence now follows from a straight forward computation
    unraveling the definition of the representation $\TwistedAction$.
  \end{proof}

  \begin{proposition}\label{restriction}
    The restriction homomorphism
    \(
      \Restriction\mapcolon
      \FinAddMeasuresOn{\PathSet}
      \longrightarrow
      \FinAddMeasuresOn{\TreeBoundary}
    \)
    restricts to an isomorphism from
    \(
      \Equalizer{\DualAlgTf}{\TheTwist\Identity}
    \)
    to
    \(
      \TwistedInvariants[\PiOne]{\FinAddMeasuresOn{\TreeBoundary}}
    \).
  \end{proposition}
  \begin{proof}
    First assume that 
    \(
      \TheMeasure
      \in
      \Equalizer{\DualAlgTf}{\TheTwist\Identity}
    \)
    and consider an oriented edge $\TreeEdge$ in $\TheTree$ pointing
    away from $\TreeBaseVertex$. Then there is a unique edge path
    $\TreeCode\turn\TreeEdge$ upstairs starting at $\TreeBaseVertex$
    ending in $\TreeEdge$ without backtracking. It projects to a postal
    code $\TheCode\turn\TheEdge$ downstairs.

    Now, consider the restriction
    $\TreeMeasure:=\RestrictionOf{\TheMeasure}$.  As $\TheMeasure$
    satisfies the identity~\pref{action-at-a-distance}, we find
    \[
      \TreeMeasureOf{\ForwardBoundaryOf{\TreeEdge}}
      =
      \TheMeasureOf{\TheCode\turn\TheEdge}
      =
      \TheTwist^{-\LengthOf{\TheCode}}
      \TheMeasureOf{\TheEdge}
    \]
    where $\LengthOf{\TheCode}$ denotes the number of edges\,/\,digits
    in $\TheCode$. Now Lemma~\ref{inv-equiv} implies
    \(
      \TreeMeasure
      \in
      \TwistedInvariants[\PiOne]{\FinAddMeasuresOn{\TreeBoundary}}
    \).
    
    Being an eigenvector for $\DualAlgTf$, the measure $\TheMeasure$ is
    completely determined by its values on oriented edges. Moreover, using
    again the identity~\eqref{action-at-a-distance}, we see that
    $\TheMeasure$ is determined by its values on any finite set of postal
    codes, provided every oriented edge occurs at least once as a final
    digit. As $\TheGraph$ does not have dead ends, $\TheMeasure$ is
    therefore determined by its values on the island designated by
    $\BaseVertex$. Hence the restriction homomorphism is injective on the
    eigenspace $\Equalizer{\DualAlgTf}{\TheTwist\Identity}$.

    It remains to show that every $\PiOne$-periodic measure
    \(
      \TreeMeasure
      \in
      \TwistedInvariants[\PiOne]{\FinAddMeasuresOn{\TreeBoundary}}
    \)
    is the restriction of an eigenmeasure
    \(
      \TheMeasure
      \in
      \Equalizer{\DualAlgTf}{\TheTwist\Identity}
    \).
    We define a function $\FnEdge$ on the set $\EdgeSet$ of
    oriented edges in $\TheGraph$ in terms of a given $\TreeMeasure$ as
    follows. For an oriented edge $\TheEdge$ choose a postal code
    $\TheCode\turn\TheEdge$ ending in $\TheEdge$ that designates a
    district in the island designated by $\BaseVertex$. Identifying
    $\TreeBoundary$ with $\PathSet[\BaseVertex]$, we write
    $\TreeMeasureOf{\TheCode\turn\TheEdge}$ for the value of the district
    $\PathSet[\TheCode\turn\TheEdge]$ regarded as a measurable set in
    $\TreeBoundary$. Finally, we put
    \(
      \FnEdgeOf{\OppositeOf{\TheEdge}}
      :=
      \TheTwist^{\LengthOf{\TheCode}}
      \TreeMeasureOf{\TheCode\turn\TheEdge}
    \).

    To show that $\FnEdge$ is well-defined, we need to consider
    an alternative postal code $\AltCode\turn\TheEdge$ also ending in
    $\TheEdge$ and also located in the island of $\BaseVertex$. The paths
    $\TheCode\turn\TheEdge$ and $\AltCode\turn\TheEdge$ both issue from
    $\BaseVertex$ and lift to the universal cover ending in two edges
    $\TreeEdge[\TheCode]$ and $\TreeEdge[\AltCode]$ pointing away from the
    base vertex $\TreeBaseVertex$ upstairs. The fundamental group acts
    transitively on the lifts of $\TheEdge$, whence there is
    $\TreeAutom\in\PiOne$ with
    $\TreeAutom\TreeEdge[\TheCode]=\TreeEdge[\AltCode]$. It follows that
    $\TreeEdge[\TheCode]$ also points away from $\InvAutom\TreeBaseVertex$
    and Lemma~\ref{inv-equiv} implies that
    $\FnEdgeOf{\OppositeOf{\TheEdge}}$ does not depend on the choice
    of the prefix $\TheCode$.

    Using additivity of $\TreeMeasure$, we compute
    \[
      (\TurnSum\FnEdge)(\OppositeOf{\TheEdge})
      =
      \Sum[\TheEdge\turn\AltEdge]{
        \FnEdgeOf{\OppositeOf{\AltEdge}}
      }
      =
      \Sum[\TheEdge\turn\AltEdge]{
        \TheTwist^{\LengthOf{\TheCode}+1}
        \TreeMeasureOf{\TheCode\turn\TheEdge\turn\AltEdge}
      }
      =
      \TheTwist^{\LengthOf{\TheCode}+1}
      \TreeMeasureOf{\TheCode\turn\TheEdge}
      =
      \TheTwist
      \FnEdgeOf{\OppositeOf{\TheEdge}}
    \]
    and apply Theorem~\ref{dual-main} to deduce that the measure
    $\TheMeasure$ induced by $\FnEdge$ is an eigenmeasure for
    $\DualAlgTf$ and restricts to $\TreeMeasure$ on $\TreeBoundary$.
  \end{proof}

  \begin{theorem}\label{thm:spectral correspondence via covering}
    For $\TheTwist\not\in\SetOf{-1,0,1}$, the isomorphisms of
    Theorem~\ref{dual-main}, Proposition~\ref{algebraic},
    Observation~\ref{equiv-isom}, Observation~\ref{twisted-isom} and
    Proposition~\ref{restriction} fit into a commutative diagram, as
    follows:
    \[
      \begin{tikzcd}
        \Equalizer{\DualAlgTf}{\TheTwist\Identity}
        \arrow[d,"\Restriction"',"\ref{restriction}"]
        \arrow[r,"\ref{dual-main}"']
        &
        \Equalizer{\TurnSum}{\TheTwist\Identity}
        \arrow[r,"\InSum","\ref{algebraic}"']
        &
        \Equalizer{\NeighborSum}{\Rescale[\TheTwist]}
        \makebox[0pt][l]{$=
          \Equalizer{\NeighborAvg}{\LocalTweak[\TheTwist]}$}
        \arrow[d,"\mathrm{lift}","\ref{equiv-isom}"']
        \\
        \TwistedInvariants[\PiOne]{
          \FAMeasuresOn{\TreeBoundary}
        }
        \arrow[rr,"{\PoissonTf[\TheTwist]}"',"\ref{twisted-isom}"]
        &
        &
        \Invariants[\PiOne]{
          \Equalizer{\TreeNeighborSum}{\TreeRescale[\TheTwist]}
        }
        \makebox[0pt][l]{$=
          \Invariants[\PiOne]{
            \Equalizer{\TreeNeighborAvg}{\TreeLocalTweak[\TheTwist]}
          }$}        
      \end{tikzcd}
    \]
  \end{theorem}
  \begin{proof}
    First note that we may read Proposition~\ref{algebraic} in reverse as
    \[
      \frac{\TheTwist}{\TheTwist^2-1}\Gradient[\TheTwist]
      =
      \InSum^{-1}
    \]
    and pick $\FnVert\in\Equalizer{\NeighborSum}{\Rescale[\TheTwist]}$
    in the upper right hand corner. Its image on the left is the measure
    $\TheMeasure$ that evaluates on postal codes as follows:
    \[
      \TheMeasureOf{\TheCode\turn\TheEdge}
      =
      \TheTwist^\LengthOf{\TheCode}
      \frac{\TheTwist}{\TheTwist^2-1}
      (\Gradient[\TheTwist]\FnVert)(\OppositeOf{\TheEdge}).
    \]
    Its restriction upstairs which appears on the lower left of
    the diagram is given by the same formula: one just restricts to those
    $\TheCode\turn\TheEdge$ starting at the base vertex $\BaseVertex$.

    Comparing this formula to the identity~\pref{boundary} that
    describes the inverse of the Poisson transform, one sees that the
    diagram commutes.
  \end{proof}

  \printbibliography

  \bigskip
  \ContactInfo
\end{document}